\documentclass[a4paper,10pt]{amsart}

\usepackage[utf8]{inputenc}
\usepackage{stackrel} 
\usepackage{graphics,graphicx} 	
\usepackage{color} 		
\usepackage{amssymb,amsmath}	
\usepackage[all]{xy}		
\usepackage{subfigure}		
\usepackage{longtable}		
\usepackage{epstopdf}		
\usepackage{pgf}
\usepackage{tikz}		
\usetikzlibrary{calc}		
\usetikzlibrary{arrows,scopes}
\usepackage{hyperref}		

\numberwithin{equation}{section}	
\newtheorem*{introtheorem}{Theorem}
\newtheorem{theorem}{Theorem}[section]
\newtheorem{lemma}[theorem]{Lemma}
\newtheorem{proposition}[theorem]{Proposition}
\newtheorem{corollary}[theorem]{Corollary}

\theoremstyle{definition}
\newtheorem{definition}[theorem]{Definition}
\newtheorem{example}[theorem]{Example}
\newtheorem{remark}[theorem]{Remark}
\newtheorem{construction}[theorem]{Construction}

\newcommand{\KK}{\mathbb K}
\newcommand{\PP}{\mathbb P}
\newcommand{\QQ}{\mathbb Q}
\newcommand{\ZZ}{\mathbb Z}

\newcommand{\seite}{\preceq}
\newcommand{\Mat}{{\rm Mat}}

\newcommand{\gq}{/\!\! /}
\newcommand{\Spec}{\mathrm{Spec}}

\renewcommand{\rho}{\varrho}

\renewcommand{\phi}{\varphi}
\renewcommand{\epsilon}{\varepsilon}

\newcommand{\subrot}[1]{{\rotatebox{#1}{\scriptsize $\subseteq$}}}

\def\gitlim{
\!\!\!
\raisebox{0.09cm}{
\scalebox{.47}{
$
\stackbin[{\rm Lim}]{\hphantom{...}\rm GIT}{\hphantom{.}\boldsymbol{/}}
$
}
}
\!
}

\def\nlimquot{
\!\!\!
\raisebox{0.09cm}{
\scalebox{.47}{
$
\stackbin[{\rm LQ}]{\hphantom{...}\sim}{\hphantom{.}\boldsymbol{/}}
$
}
}
\!
}

\def\limquot{
\!\!\!
\raisebox{0.09cm}{
\scalebox{.47}{
$
\stackbin[{\rm LQ}]{\hphantom{...}}{\hphantom{.}\boldsymbol{/}}
$
}
}
\!
}

\begin{document}

\title[Point Configurations and Translations]%
{Point Configurations and Translations}

\author[H.~B\"aker]{Hendrik B\"aker}

\email{hendrik.baeker@uni-tuebingen.de}
\address{Mathematisches Institut,
Universit\"at T\"ubingen,
Auf der Morgenstelle 10,
72076 T\"u\-bingen,
Germany}


\begin{abstract}
The spaces of point configurations on the projective line up to the action of $\mathrm{SL}(2,\KK)$ and its maximal torus are canonically compactified by the Grothdieck-Knudsen and Losev-Manin moduli spaces $\overline M_{0,n}$ and $\overline L_n$ respectively.
We examine the configuration space up to the action of the maximal unipotent group $\mathbb G_a\subseteq\mathrm{SL}(2,\KK)$ and define an analogous compactification.
For this we first assign a canonical quotient to the action of a unipotent group on a projective variety. Moreover, we show that similar to $\overline M_{0,n}$ and $\overline L_n$ this quotient arises in a sequence of blow-ups from a product of projective spaces.
\end{abstract}

\maketitle

\def\gitfani{

\begin{tikzpicture}
  \fill[gray!30] (1,0) -- (0.5,0.88) -- (1.5,0.88) -- cycle;
  \draw (0,0) -- (2,0) -- (1,1.72) -- cycle;
  \draw (1,0) -- (0.5,0.88) -- (1.5,0.88) -- cycle;
  
  \fill (0,0) circle (.05cm);
  \node [below left] at (0,0) {{$w_{01}$}};
  
  \fill (1,0) circle (.05cm);
  \node [below] at (1,0) {{$w_{12}$}};
  
  \fill (2,0) circle (.05cm);
  \node [below right] at (2,0) {{$w_{02}$}};
  
  \fill (1,1.72) circle (.05cm);
  \node [above] at (1,1.72) {{$w_{03}$}};
  
  \fill (0.5,0.88) circle (0.05cm);
  \node [left] at (0.5,0.88) {{$w_{13}$}};
  
  \fill (1.5,0.88) circle (0.05cm);
  \node [right] at (1.5,0.88) {{$w_{23}$}};
  \end{tikzpicture}
}


\def\gitfanii{ 
  
\begin{tikzpicture}
	\def\ang{45};
	
	\def\x{0};
	\def\y{0};
	\def\z{0};
	\pgfmathsetmacro\xcor{\x+sin(\ang)*\y};
	\pgfmathsetmacro\ycor{\z+cos(\ang)*\y};
	\coordinate (w01) at ($(\xcor,\ycor)$);
        \node [below left] at (w01) {{$w_{01}$}};
        
        \def\x{2};
	\def\y{0};
	\def\z{0};
	\pgfmathsetmacro\xcor{\x+sin(\ang)*\y};
	\pgfmathsetmacro\ycor{\z+cos(\ang)*\y};
	\coordinate (w02) at ($(\xcor,\ycor)$);
        \node [below right] at (w02) {{$w_{02}$}};
        
        \def\x{1};
	\def\y{1.732};
	\def\z{0};
	\pgfmathsetmacro\xcor{\x+sin(\ang)*\y};
	\pgfmathsetmacro\ycor{\z+cos(\ang)*\y};
	\coordinate (w03) at ($(\xcor,\ycor)$);
        \node [right] at (w03) {{$w_{03}$}};
        
        \def\x{1};
	\def\y{0.577};
	\def\z{1.732};
	\pgfmathsetmacro\xcor{\x+sin(\ang)*\y};
	\pgfmathsetmacro\ycor{\z+cos(\ang)*\y};
	\coordinate (w04) at ($(\xcor,\ycor)$);
        \node [above] at (w04) {{$w_{04}$}};

	\def\x{1};
	\def\y{0};
	\def\z{0};
	\pgfmathsetmacro\xcor{\x+sin(\ang)*\y};
	\pgfmathsetmacro\ycor{\z+cos(\ang)*\y};
	\coordinate (w12) at ($(\xcor,\ycor)$);
        \node [below] at (w12) {{$w_{12}$}};
        
       	\def\x{0.5};
	\def\y{0.866};
	\def\z{0};
	\pgfmathsetmacro\xcor{\x+sin(\ang)*\y};
	\pgfmathsetmacro\ycor{\z+cos(\ang)*\y};
	\coordinate (w13) at ($(\xcor,\ycor)$);
        
        \def\x{0.5};
	\def\y{0.289};
	\def\z{0.866};
	\pgfmathsetmacro\xcor{\x+sin(\ang)*\y};
	\pgfmathsetmacro\ycor{\z+cos(\ang)*\y};
	\coordinate (w14) at ($(\xcor,\ycor)$);
        \node [left] at (w14) {{$w_{14}$}};
        
        \def\x{1.5};
	\def\y{0.866};
	\def\z{0};
	\pgfmathsetmacro\xcor{\x+sin(\ang)*\y};
	\pgfmathsetmacro\ycor{\z+cos(\ang)*\y};
	\coordinate (w23) at ($(\xcor,\ycor)$);
        \node [right] at (w23) {{$w_{23}$}};
        
        \def\x{1.5};
	\def\y{0.289};
	\def\z{0.866};
	\pgfmathsetmacro\xcor{\x+sin(\ang)*\y};
	\pgfmathsetmacro\ycor{\z+cos(\ang)*\y};
	\coordinate (w24) at ($(\xcor,\ycor)$);
        
        \def\x{1};
	\def\y{1.155};
	\def\z{0.866};
	\pgfmathsetmacro\xcor{\x+sin(\ang)*\y};
	\pgfmathsetmacro\ycor{\z+cos(\ang)*\y};
	\coordinate (w34) at ($(\xcor,\ycor)$);
        \node [right] at (w34) {{$w_{34}$}};

    \fill[gray!30] (w12) -- (w14) -- (w34) -- (w23) --cycle;    
    
    \fill (w01) circle (0.05cm);
    \fill (w02) circle (0.05cm);
    \fill (w03) circle (0.05cm);
    \fill (w04) circle (0.05cm);
    \fill (w12) circle (0.05cm);
    \fill (w13) circle (0.05cm);
    \fill (w14) circle (0.05cm);
    \fill (w23) circle (0.05cm);
    \fill (w24) circle (0.05cm);
    \fill (w34) circle (0.05cm);
    
    \draw[thick] (w01) -- (w02) -- (w04) --cycle;
    \draw[thick] (w04) -- (w03) -- (w02);
    \draw[dashed] (w03) -- (w01);
    
    \draw[thick] (w14) -- (w24) -- (w23);
    \draw[dashed] (w23) -- (w13) -- (w14);
        
    \draw[thick] (w12) -- (w24) -- (w34);
    \draw[dashed] (w34) -- (w13) -- (w12);
    
    \draw[thick] (w12) -- (w14);
    \draw[thick] (w34) -- (w23);
    \draw[dashed] (w14) -- (w34);
    \draw[dashed] (w12) -- (w23);
    
   \end{tikzpicture}
}
 

\def\gitfaniii{ 
 
 \begin{tikzpicture}
	\def\ang{35};

	\def\x{1};
	\def\y{0};
	\def\z{0};
	\pgfmathsetmacro\xcor{\x+sin(\ang)*\y};
	\pgfmathsetmacro\ycor{\z+cos(\ang)*\y};
	\coordinate (w12) at ($(\xcor,\ycor)$);
                
       	\def\x{0.5};
	\def\y{0.866};
	\def\z{0};
	\pgfmathsetmacro\xcor{\x+sin(\ang)*\y};
	\pgfmathsetmacro\ycor{\z+cos(\ang)*\y};
	\coordinate (w13) at ($(\xcor,\ycor)$);
                
        \def\x{0.5};
	\def\y{0.289};
	\def\z{0.866};
	\pgfmathsetmacro\xcor{\x+sin(\ang)*\y};
	\pgfmathsetmacro\ycor{\z+cos(\ang)*\y};
	\coordinate (w14) at ($(\xcor,\ycor)$);
                
        \def\x{1.5};
	\def\y{0.866};
	\def\z{0};
	\pgfmathsetmacro\xcor{\x+sin(\ang)*\y};
	\pgfmathsetmacro\ycor{\z+cos(\ang)*\y};
	\coordinate (w23) at ($(\xcor,\ycor)$);
                
        \def\x{1.5};
	\def\y{0.289};
	\def\z{0.866};
	\pgfmathsetmacro\xcor{\x+sin(\ang)*\y};
	\pgfmathsetmacro\ycor{\z+cos(\ang)*\y};
	\coordinate (w24) at ($(\xcor,\ycor)$);
                
        \def\x{1};
	\def\y{1.155};
	\def\z{0.866};
	\pgfmathsetmacro\xcor{\x+sin(\ang)*\y};
	\pgfmathsetmacro\ycor{\z+cos(\ang)*\y};
	\coordinate (w34) at ($(\xcor,\ycor)$);
        
    \fill[gray!30] (w12) -- (w14) -- (w34) -- (w23) --cycle;

    \fill (w12) circle (0.05cm);
    \fill (w13) circle (0.05cm);
    \fill (w14) circle (0.05cm);
    \fill (w23) circle (0.05cm);
    \fill (w24) circle (0.05cm);
    \fill (w34) circle (0.05cm);
    
    \draw[thick] (w14) -- (w12) -- (w23) -- (w34) -- cycle;
    \draw (w23) -- (w13) -- (w14);
        
    \draw[thick] (w12) -- (w24) -- (w34);
    \draw (w34) -- (w13) -- (w12);
    
    \draw[thick] (w14) -- (w24) -- (w23);
    \draw (w14) -- (w13) -- (w23);

    \draw[dashed] (w12) -- (w34);
    \draw[dashed] (w13) -- (w24);
    \draw[dashed] (w14) -- (w23);
    
    \node [below] at (w12) {{$w_{12}$}};
    \node [left] at (w13) {{$w_{13}$}};
    \node [left] at (w14) {{$w_{14}$}};
    \node [right] at (w23) {{$w_{23}$}};
    \node [right] at (w24) {{$w_{24}$}};
    \node [right] at (w34) {{$w_{34}$}};
    
   \end{tikzpicture}
}


\def\blowupi{
\begin{tikzpicture}
  \draw (0,0) -- (1,0) -- (0.5,0.86) -- cycle;
  \draw (0.5,0) -- (0.5,0.86);
  \draw (0.5,0) -- (0.75,0.43);
  
  \fill (0,0) circle (.05cm);
  \node [below left] at (0,0) {{$e_1$}};
  \fill (1,0) circle (.05cm);
  \node [below right] at (1,0) {{$e_2$}};
  \fill (0.5,0.86) circle (.05cm);
  \node [above] at (0.5,0.86) {{$e_3$}};
  \fill (0.5,0) circle (.05cm);
  \node [below] at (0.5,0) {{$\nu_1$}};
  \fill (0.75,0.43) circle (0.05cm);
  \node [right] at (0.75,0.43) {{$\nu_2$}};
 \end{tikzpicture}
 }

 \def\blowupii{
 \begin{tikzpicture}
  \draw (0,0) -- (1,0) -- (0.5,0.86) -- cycle;
  \draw (0,0) -- (0.75,0.43);
  \draw (0.5,0) -- (0.75,0.43);
  
  \fill (0,0) circle (.05cm);
  \node [below left] at (0,0) {{$e_1$}};
  \fill (1,0) circle (.05cm);
  \node [below right] at (1,0) {{$e_2$}};
  \fill (0.5,0.86) circle (.05cm);
  \node [above] at (0.5,0.86) {{$e_3$}};
  \fill (0.5,0) circle (.05cm);
  \node [below] at (0.5,0) {{$\nu_1$}};
  \fill (0.75,0.43) circle (0.05cm);
  \node [right] at (0.75,0.43) {{$\nu_2$}};
 \end{tikzpicture}
 }

 \def\blowupiii{
 \begin{tikzpicture}
  \draw (0,0) -- (1,0) -- (0.5,0.86) -- cycle;
  \draw (0.5,0) -- (0.5,0.86);
  \draw (0,0) -- (0.75,0.43);
  \draw (0.5,0.29) -- (1,0);
  
  \fill (0,0) circle (.05cm);
  \node [below left] at (0,0) {{$e_1$}};
  \fill (1,0) circle (.05cm);
  \node [below right] at (1,0) {{$e_2$}};
  \fill (0.5,0.86) circle (.05cm);
  \node [above] at (0.5,0.86) {{$e_3$}};
  \fill (0.5,0) circle (.05cm);
  \node [below] at (0.5,0) {{$\nu_1$}};
  \fill (0.75,0.43) circle (0.05cm);
  \node [right] at (0.75,0.43) {{$\nu_2$}};
  \fill (0.5,0.29) circle (0.05cm);
 \end{tikzpicture}
 }

\section{Introduction}


In the present paper we examine point configurations on the projective line up to translations. In general, let us consider $n$ distinct points on $\PP_1$. Then the open subset $U\subseteq \PP_1^n$ consisting of pairwise different coordinates
is the space of possible configurations. For an algebraic group $G$ acting on $\PP_1$ the question arises what the resulting equivalence classes of configurations are, i.e. we ask for a quotient $U/G$ of the diagonal action and a possible canonical compactification. 

In the case of the full automorphism group $G=\mathrm{SL}(2,\KK)$ this problem has been thoroughly studied. The space of configuration classes is canonically compactified by the famous Grothendieck-Knudsen moduli space $\overline M_{0,n}$, i.e. we have
\[
 M_{0,n}
 \ =\ 
 U\,/\,\mathrm{SL}(2,\KK)
 \ \subseteq\ 
 \overline M_{0,n}.
\]
Originally introduced as moduli space of certain marked curves Kapranov shows in \cite{chow_quotients_of_grassmannians} that $\overline M_{0,n}$ has (among others) the following two equivalent descriptions. Firstly it arises as the GIT-limit of $\PP_1^n$ with respect to the $G$-action, i.e. the limit of the inverse system of Mumford quotients. Secondly, it can be viewed as the blow-up of $\PP_{n-3}$ in $n-1$ general points and all the linear subspaces of dimension at most $n-5$ spanned by them.

Later this setting has been studied in the case where the full automorphism group was replaced by its maximal torus $\KK^*\subseteq\mathrm{Sl}(2,\KK)$. Similarly, it turns out that the Losev-Manin moduli space $\overline L_n$ coincides with the the GIT-limit, which in this case is the toric variety associated to the permutahedron. Again, the GIT-limit arises in a sequence of (toric) blow-ups from projective space, see \cite{discriminants, quotients_of_toric_varieties, losev_manin}.

In this paper we treat point configurations on $\PP_1$ up to the action of the maximal connected unipotent subgroup $\mathbb G_a\subseteq\mathrm{SL}(2,\KK)$. It consists of upper triangular matrices with diagonal elements equal to $1_\KK$ and can be thought of as group of translations. Since $\mathbb G_a$ is not reductive, we are faced with the additional problem of first finding to suitable replacement for the GIT-limit, i.e. assigning a canonical quotient to this action. This will be overcome in the following manner.

Doran and Kirwan introduce in \cite{towards_non_reductive_git} the notion of finitely generated semistable points admitting so-called enveloped quotients. Moreover, in \cite{non_reductive_gelfand} Arzhantsev, Hausen and Celik propose a Gelfand-MacPherson type construction which allows to apply methods from reductive GIT to obtain these enveloped quotients. Building on this work we obtain again an inverse system and the corresponding GIT-limit. In general the enveloped quotients are not projective, hence one cannot expect the GIT-limit to be so.

We then show that (up to nomalisation) the limit quotient, i.e. a canonical component of the GIT-limit, is canonically compactified by an iterated blow-up of $\PP_1^{n-1}$. To make this a little more precise consider a subset $A\subseteq \{2,\ldots, n\}$. Denoting by $T_2,S_2,\ldots, T_n,S_n$ the homogeneous coordinates on $\PP_1^{n-1}$ we associate to $A$ a subscheme $X_A$ on $\PP_1^{n-1}$ given by the ideal
\[
 \big\langle \,T_i^2,\ T_jS_k-T_kS_j;\;i,j,k\in A,\;j<k\,\big\rangle.
\]
The scheme-theoretic inclusions give rise to a partial order of these subschemes. Let $\mathrm{Bl}(\PP_1^{n-1})$ denote the blow-up of $\PP_1^{n-1}$ in all these subschemes in non-descending order.
\begin{introtheorem}
If $\PP_1\nlimquot \mathbb G_a$ and $\tilde{\mathrm{Bl}}(\PP_1^{n-1})$ denote the normalisations of the limit quotient and the above blow-up of $\PP_1^{n-1}$ respectively, then we have open embeddings
\[
 U/\mathbb G_a
 \quad\subseteq\quad
 \PP_1^n\nlimquot \mathbb G_a
 \quad\subseteq\quad
 \tilde{\mathrm{Bl}}(\PP_1^{n-1}).
\]
\end{introtheorem}

In the case of two distinct points, i.e. $n=2$, the latter space is simply $\PP_1$. If $n=3$ holds, then the compatification $\tilde{\mathrm{Bl}}(\PP_1\times\PP_1)$ is the unique non-toric, Gorenstein, log del Pezzo $\KK^*$-surface of Picard number $3$ with a singularity of type $A_1$. Similar to $\overline M_{0,5}$ which arises as a single Mumford quotient of the cone over the Grassmannian $\mathrm{Gr}(2,5)$, this surface is the Mumford quotient of the cone over the Grassmannian $\mathrm{Gr}(2,4)$. For higher $n$ an analogous Mumford quotient needs to be blown up as will be described in Section~\ref{sec:iterated_blow-up}.

The paper is organized as follows. In Section~\ref{sec:nr_limit} we recall the results of \cite{non_reductive_gelfand} and introduce the non-reductive GIT-limit and limit quotient. In the following Section~\ref{sec:N0m} we apply these constructions to the action of $\mathbb G_a$ on $\PP_1^n$. We discuss explicitly the GIT-fan which contains the combinatorial data needed to make the limit quotient accessible.
The blow-ups of $\PP_1^{n-1}$ will be dealt with in a mostly combinatorial way, i.e. as proper transforms with respect to toric blow-ups. For this we prove a result on combinatorial blow-ups in the spirit of Feichtner and Kozlov, see \cite{incidence_combinatorics}. This will be carried out in Section~\ref{sec:combinatorics}. The final Section~\ref{sec:iterated_blow-up} then is dedicated to the proof of the main theorems.
\tableofcontents

\section{The non-reductive GIT-limit}
\label{sec:nr_limit}
In this section we deal with the problem of assigning a canonical quotient to the action of a unipotent group $G$ on a Mori Dream Space $X$, i.e. a $\QQ$-factorial, projective variety with finitely generated Cox ring $\mathcal R(X)$. For reductive groups an answer to this problem is the GIT-limit, i.e. the limit of the inverse system consisting of the Mumford quotients $X^{\mathrm{ss}}(D)\gq G$. However, this method relies on Hilbert's Finiteness Theorem which guarantees, that for a linear action of a reductive group $G$ on any affine algebra the invariant algebra is affine again. So we make a further finiteness assumption on certain $G$-invariants which for example holds when $G=\mathbb{G}_a$.

In \cite[Definition~4.2.6]{towards_non_reductive_git} Doran and Kirwan introduce the notion of {\it finitely generated semistable sets} for the action of a unipotent group, namely the sets $X_{\mathrm{fg}}^{\mathrm{ss}}(D):=\bigcup X_f$ where $D$ is some ample divisor, $f\in \mathcal O_{n\!D}(X)^G$ is an invariant section for some $n>0$ and $\mathcal O(X_f)^G$ is finitely generated. These sets possess {\it enveloped quotients }
\[
 r\colon X^{\mathrm{ss}}_\mathrm{fg}(D)\ \to \ r\,(X^\mathrm{ss}_{\mathrm {fg}}(D))\subseteq X\gq_{\!D} G
\]
where the {\it enveloping quotient} $X\gq_{\!D} G$ can be obtained by gluing together the affine varieties $\Spec(\mathcal O(X_f)^G)$. 
Using a Gelfand-MacPherson type correspondence described in \cite{non_reductive_gelfand} we now turn this collection of enveloped quotients into an inverse system.


Consider the action of an affine-algebraic, simply connected group $G$ with trivial character group $\mathbb X(G)$ on the normal, projective variety $X$. Let $K\subseteq\mathrm{WDiv}(X)$ be a free and finitely generated group of Weil divisors mapping isomorphically onto the divisor class group $\mathrm{Cl}(X)$. We then associate to $X$ a sheaf of graded algebras
\[
 \mathcal R\;\;\;:=\bigoplus_{D\in K}\mathcal O_D.
\]
We suppose that the algebra of global sections $\mathcal R(X)$, i.e. the Cox ring of $X$, is finitely generated. The $K$-grading yields an action of the torus $H:=\Spec(\KK[K])$ on the relative spectrum $\hat X:=\mathrm{Spec}_X(\mathcal R)$ and the canonical morphism $p\colon\hat X\to X$ is a good quotient for this action. By linearisation the $G$-action on $X$ lifts to a unique action of $G$ on the total coordinate space $\overline{X}:=\Spec(\mathcal R(X))$ which commutes with the $H$-action and turns $p$ into an equivariant morphism, see \cite[Section~1]{git_based_on_weil_divisors}.

Now suppose that the algebra of invariants $\mathcal R(X)^G$ is finitely generated as well and let $\overline{Y}$ be its spectrum. The inclusion of the invariants gives rise to a morphism $\kappa\colon \overline{X}\to\overline{Y}$. Since $\kappa$ is not necessarily surjective, it need not have the universal property of quotients. However, passing to the category of constructible spaces we obtain a categorical quotient $\kappa\colon \overline X\to \overline Y':=\kappa(\overline X)$, see \cite{non_reductive_gelfand} for details.

For every ample $D\in K$ standard geometric invariant theory provides us with a set of semistable points 
\[\overline Y^{\mathrm {ss}}(D):=\bigcup\; \overline{Y}_f
 \quad \text{where} \quad
 f\in\mathcal R(X)_{nD}^G
 \quad \text{and}\quad
 n>0. 
\]
These sets admit good quotients for the $H$-action which are isomorphic to the enveloping quotient $X\gq_{\!D}G$ in the sense of Doran and Kirwan. The set of finitely generated semistable points $X^{\mathrm {ss}}_{\mathrm{fg}}(D)$ can be retrieved from $\overline Y^{\mathrm {ss}}(D)$ by
\[
 X^{\mathrm {ss}}_{\mathrm{fg}}(D)=p(\hat U)\qquad\text{where}\qquad \hat U:=\kappa^{-1}(\overline Y^{\mathrm {ss}}(D)).
\]
The situation fits into the following commutative diagram:
\[
\xymatrix{
\overline{X}
\ar@/^2pc/[rrrrr]!/u 1.4pc/^{\kappa}
&
\hat{X}
\ar@{}[l]|{\supseteq}
\ar[d]^p
&
\hat{U}
\ar@{}[l]|{\supseteq}
\ar[rr]^{\kappa}
\ar[d]^p
&
&
{\overline{Y}^{\mathrm{ss}}(D)\cap\overline Y'}
\ar@{}[r]!/u 1.3pc/|-{\subrot{22}}
\ar@{}[r]!/d 1pc/|-{\subrot{-22}}
\ar[dd]^{q}
&
{\begin{array}{c}
    {\overline Y'}
    \\
    \\
    {\overline{Y}^{\mathrm{ss}}(D)}
 \end{array}
}
\ar[dd]^{q'}
&
{\overline Y}
\ar@{}[l]!/u 1.3pc/|-{\subrot{-22}}
\ar@{}[l]!/d 1pc/|-{\subrot{22}}
\\
&
X
&
{X^{\mathrm{ss}}_{\mathrm{fg}}(D)}
\ar@{}[l]|{\supseteq}
\ar[d]^r
\\
&&V
&&
V
\ar@{=}[ll]
\ar@{}[r]|{\subseteq}
&
{
\begin{array}{c}
X\gq_{\!D}G\ =
\\
\overline Y^{\mathrm{ss}}(D)\gq H
\end{array}
}
}
\]
In this setting \cite[Corollary~5.3]{non_reductive_gelfand} answers the question whether the morphisms $q$ and $r$ are categorical quotients.
\begin{proposition}
\label{pro:true_quotients}
 If for every $v\in V$ the closed $H$-orbit lying in $q'^{-1}(v)$ is contained in $\overline Y'$ (e.g. $q'$ is geometric), then $q$ and $r$ are categorical quotients for the $H$- and $G$-actions respectively.
\end{proposition}

In order to define a canonical quotient for the action of $G$ on $X$ we first recall the respective methods in reductive geometric invariant theory. For the affine variety $\overline Y$ let $\overline Y_1,\ldots, \overline Y_r$ be the sets of semistable points arising from ample divisors. Whenever we have $\overline Y_i\subseteq\overline Y_j$ for two of these set we obtain a commutative diagram.
\[
 \xymatrix{
 {\overline Y_i}
 \ar[r]
 \ar[d]
 &
 {\overline Y_j}
 \ar[d]
 \\
 {\overline Y_i\gq H}
 \ar[r]^{\varphi_{ij}}
 &
 {\overline Y_j\gq H}
 }
\]
The morphisms $\varphi_{ij}\colon \overline Y_i\gq H\to\overline Y_j\gq H$ turn the collection of quotients into an inverse system, the {\it GIT-system}. Its inverse limit $\overline Y\gitlim H$ is called {\it GIT-limit}. There exists a canonical morphism
\[
 \bigcap \overline Y_i\quad \to\quad \overline Y\gitlim H
\]
and the closure of its image is the {\it limit quotient} $\overline Y\limquot H$ of $\overline Y$ with respect to $H$. Note that in the literature this space is also called 'canonical component' or 'GIT-limit'. In general, the limit quotient need not be normal; its normalisation is the {\it normalised limit quotient} $\overline Y\nlimquot H$.

We now turn to the non-reductive case. As constructible subsets of $\overline Y_i\gq H$ the corresponding enveloped quotients $V_i$ inherit the above morphisms $\varphi_{ij}$, and again form an inverse system. 

\begin{definition}
The {\it (non-reductive) GIT-limit} $X\gitlim G$ of $X$ with respect to the $G$-action is the limit of the inverse system of enveloped quotients. 
\end{definition}

The non-reductive GIT-limit $X\gitlim G$ is a constructible subset of the reductive GIT-limit $\overline Y\gitlim H$. Analogously, we obtain a canonical morphism into the (non-reductive) GIT-limit $X\gitlim G$
\[
 \bigcap\, (\overline Y'\cap \overline Y_i) \quad \to\quad X\gitlim G.
\]

\begin{definition}
 The {\it (non-reductive) limit quotient} $X\limquot G$ of $X$ with respect to the $G$-action is the closure of the image of the above morphism. Its normalisation is the {\it normalised limit quotient} $X\nlimquot G$.
\end{definition}

The limit quotient in general appears to be relatively hard to access. However, if $\overline Y$ is factorial we can realise it up to normalisation as a certain closed subset of a toric variety as follows. For this consider homogeneous generators $f_1,\ldots, f_r$ of the $K$-graded algebra $\mathcal O(\overline Y)$. With $\deg(T_i):=\deg (f_i)$ we obtain a graded epimorphism
\[
 \KK[T_1,\ldots, T_r]\ \to\ \mathcal O(\overline Y);\qquad T_i\ \mapsto \ f_i.
\]
This gives rise to an equivariant closed embedding of $\overline Y$ into $\KK^r$. We denote by $Q$ the the matrix recording the weights $\deg(f_i)$ as columns and fix a Gale dual matrix $P$, i.e. a matrix with $PQ^t=0$. The {\it Gelfand-Kapranov-Zelevinsky-decomposition (GKZ-decomposition)} of $P$ is the fan 
\[
 \Sigma\ :=\ \{\sigma(v);~v\in\QQ^{r-\mathrm{rk}(K)}\},\qquad
 \sigma(v)\ :=\ \bigcap_{v\in\tau^\circ}\tau
\]
where $\tau$ is a cone generated by some of the columns of $P$. It is known that the normalised limit quotient $\KK^r\nlimquot H$ is a toric variety with corresponding fan $\Sigma$. Now suppose that $\overline Y$ is factorial. Then every set of semistable points of $\overline Y$ arises as intersection of $\overline Y$ with a set of semistable points on $\KK^r$. In this situation we obtain a closed embedding of the GIT-limits $\overline Y\gitlim H\to\KK^r\gitlim H$ and hence of the respective limit quotients. The inverse image of $\overline Y\limquot H$ under the normalisation map $\nu\colon \KK^r\nlimquot H\to\KK^r\limquot H$ is in general not normal. However, its normalisation conincides with the normalised limit quotient $\overline Y\nlimquot H$. The situation fits into the following commutative diagram.
\[
 \xymatrix{
{\overline Y\nlimquot H}
 \ar[r]
 \ar[rd]
 &
 {\nu^{-1}(\overline Y\limquot H)}
 \ar[r]
 \ar[d]^\nu
 &
 {\KK^r\nlimquot H}
 \ar[d]^\nu
 \\
  &
 {\overline Y\limquot H}
 \ar[r]
 &
 {\KK^r\limquot H}
 }
\]
Finally, if $T$ is the dense torus in $\KK^r$, then $\nu^{-1}(\overline Y\limquot H)$ coincides with the closure of $(\overline Y\cap T)/H$ in $\KK^r\nlimquot H$. Hence we obtain a normalisation map 
\[
 \overline Y\nlimquot H\ \to\ \overline{\left((\overline Y\cap T)\,/\,H\right)^{\Sigma}}.
\]

\section{Point configurations on \texorpdfstring{$\PP_1$}{P1} and translations}
\label{sec:N0m}
In this section we examine point configurations on $\PP_1^n$ up to translations. For this we consider the diagonal action of $\mathbb G_a$ on $\PP_1^n$ and explicitly perform the Gelfand-MacPherson type construction introduced in the preceding section. We determine the GIT-fan describing the variation of quotients and show that it is closely related to the well known GIT-fan stemming from the action of the full automorphism group $\mathrm{SL}(2,\KK)$ on $\PP_1^n$.

For this we consider the unipotent group
\[\mathbb{G}_a=\left\{\left(\begin{array}{cc}1&k\\0&1\end{array}\right);~k\in\KK\right\}\subseteq\mathrm{SL}(2,\KK),\]
and its action on $\overline X:=(\KK^n)^2$ given by
\[
 A\cdot \left[\begin{array}{ccc}
               x_1&\ldots&x_n\\
               y_1&\ldots&y_n
              \end{array}\right]
\ :=\ 
\left[A\left(\begin{array}{cc}x_1\\y_1\end{array}\right),\ldots,A\left(\begin{array}{cc}x_n\\y_n\end{array}\right)\right].
\]
Viewing $[x_i,y_i]$ as homogeneous coordinates of the factors in $\PP_1^n$ this gives rise to an induced action on $X:=\PP_1^n$. Note that the Cox ring of $X$ is
\[\mathcal R(X)=\mathcal O(\overline X)=\KK[T_1,\ldots, T_n,S_1,\ldots,S_n]\]
together with a $\mathrm{Cl}(X)$-grading defined by $\deg(T_i)=\deg (S_i)=e_i\in\ZZ^n=\mathrm{Cl}(X)$. A first Propositions concerns the algebra of invariants in $\mathcal R(X)$ and its spectrum.

\begin{proposition}
\label{pro:invariants}
Consider the above $\mathbb{G}_a$-action on $\overline X$.
\begin{enumerate}
 \item The subalgebra $\mathcal O(\overline X)^{\mathbb{G}_a}\subseteq\mathcal O(\overline X)$ is generated by 
\[
  S_1,\ldots, S_n, \qquad T_jS_k-T_kS_j \text{, with } 1\le j<k\le n.
 \]

\item The canonical morphism $\kappa'\colon \overline X\to\overline Y$ where $\overline Y:=\mathrm{Spec}(\mathcal O(\overline X)^{\mathbb{G}_a})$ fits into a commutative diagram
\[
 \xymatrix{
 {\overline X}
 \ar[rrrr]^{\kappa\colon (x,y)\ \mapsto\ (1,x)\,\wedge\,(0,y)}
 \ar[rrd]_{\kappa'}
 &&&&
 {\bigwedge\nolimits^2\KK^{n+1}}
 \\
 &&
 {\overline Y}
 \ar[rru]_{\iota}
 }
\]
where $\iota$ is a closed embedding and its image $\iota(\overline Y)$ is the affine cone over the Grassmannian $\mathrm{Gr}(2,n+1)$. Its vanishing ideal is generated by the Pl\"ucker relations
\[
 T_{ij}T_{kl}-T_{ik}T_{jl}+T_{il}T_{jk},\qquad\text{with}\qquad 0\le i<j<k<l\le n,
\]
where $T_{ij}=(e_i\wedge e_j)^*$ are the dual basis vectors of the standard basis.
\end{enumerate}
\end{proposition}

\begin{proof}
The invariants have been described by Shmelkin, see \cite[Theorem 1.1]{shmelkins_invariants}. For (ii) we define $\iota$ by its comorphism 
\[
 \iota^*\colon T_{0i}\mapsto S_i,\quad T_{jk}\mapsto T_jS_k-T_kS_j
 \quad
 \text{where }
 1\le i\le n,\ 1\le j< k\le n.
\]
Clearly, $\iota^*$ is surjective, hence $\iota$ is an embedding. Moreover, the pullback of the Pl\"ucker relations with $\iota^*$ gives the zero ideal. Thus $\overline Y$ lies in the affine cone $C(\mathrm{Gr}(2,n+1))$. It now suffices to show that $\mathrm{Im}(\kappa')$ has dimension $2n-1$.

For this consider two points $(x,y),\,(x',y')$ with only non-zero coefficients. If they have distinct orbits, then the orbits are separated by the invariants: If $y\not= y'$ holds, then there exists a separating $S_i$. Otherwise we can choose a separating $T_iS_j-T_jS_i$. Hence, over an open set the fibres of $\kappa'$ are one-dimensional and thus the image of $\kappa'$ is $(2n-1)$-dimensional.
\end{proof}

While for reductive groups the quotient morphism $\kappa'$ is surjective, this fails in general. We provide a description of the image of 
\[\kappa\colon\overline X=(\KK^n)^2\to\bigwedge\nolimits^2\KK^{n+1};\quad (x,y)\mapsto(1,x)\wedge (0,y).\]
Via the embedding of the preceeding proposition we view $\overline Y$ as subset of $\bigwedge^2\KK^{n+1}$. Observe that $\overline Y$ contains the affine cone $\overline Y^\star$ of the smaller Grassmannian $\mathrm{Gr}(2,n)$ in the following canonical manner:
\[\overline Y^\star=\{(0,x)\wedge (0,y);\ x,y\in\KK^n\}\subseteq\overline Y.\]

\begin{proposition}
\label{pro:Im_of_kappa}
 The image of $\kappa$ is $\kappa(\overline X)\,=\,(\overline Y\setminus \overline Y^\star)\cup\{0\}$ .
\end{proposition}

\begin{proof}
From the definition of the morphism $\kappa$ it follows that its image is contained in $(\overline Y\setminus \overline Y^\star)\cup\{0\}$. For the reverse inclusion consider 
\[z=\sum z_{ij}e_i\wedge e_j\ \in\ \overline Y\setminus \overline Y^\star.\]
We define $y:=(z_{01},\ldots,z_{0n})\in \KK^n$; note that $y\not=0$ holds. With the identification $\KK^n=\{0\}\times \KK^n\subseteq\KK^{n+1}$ we obtain an affine subspace $W_y$ by
\[
W_y\ :=\ e_0\wedge y+\bigwedge\nolimits^2\KK^n
\quad \subseteq\quad
\left(\KK e_0\bigwedge\KK^n\right)\oplus\bigwedge\nolimits^2\KK^n
\quad =\quad 
\bigwedge\nolimits^2\KK^{n+1}.
\]
Since $z$ lies in $W_y\cap \overline Y$, it suffices to show that $\kappa(\,\cdot\,,y)$ maps $\KK^n$ onto $W_y\cap \overline Y$. Clearly, by definition of $\kappa$, the image of $\kappa(\,\cdot\,,y)$ lies in $W_y\cap \overline Y$. To show surjectivity we regard $W_y$ as a vector space with origin $e_0\wedge y$. Then there is a linear map
 \[
  \varphi\colon W_y\to \bigwedge\nolimits^3\KK^n;\quad  e_0\wedge y+u\wedge v\ \mapsto\  u\wedge v\wedge y.
 \]
Observe that we have inclusions $\mathrm{Im}(\kappa(\,\cdot\,,y))\subseteq Z_y\subseteq \mathrm{ker}(\varphi)$. We claim that equality holds in both cases. Since $\kappa(\,\cdot\,,y)$ is linear of rank $n-1$, the claim follows from
 \[
 \mathrm{dim}(\mathrm{ker} (\varphi))\ =\ \mathrm{dim}(W_y)-\mathrm{rank}(\varphi)\  =\ \binom{n}{2} - \binom{n-1}{2}\ =\ n-1.
\]
\end{proof}


We recall from \cite[Section~2]{hausen:amplecone} the definition of the GIT-fan. Let the algebraic torus $H:=(\KK^*)^n$ act diagonally on $\KK^r$ via the characters $\chi^{w_1},\ldots,\chi^{w_r},\ w_i\in\ZZ^n$, i.e.
\[
 h\cdot z\ :=\ (\chi^{w_1}(h)\,z_1,\ldots,\chi^{w_r}(h)\,z_r)
\]
and suppose that $Y\subseteq\KK^r$ is invariant under this action. Then the {\it GIT-fan} is defined as the collection of cones
\[
 \Lambda_H(Y)\ :=\ \{\lambda(w);~w\in\QQ^n\};\qquad\lambda(w)\ :=\ \bigcap_{w\in \omega_I}\omega_I\ \subseteq\  \QQ^n,
\]
where $\omega_I:=\mathrm{cone}(w_i;\ i\in I)$ is the cone associated to a {\it $Y$-set} $I$, i.e. a subset $I\subseteq\{1,\ldots,r\}$ for which the corresponding stratum $\{y\in Y;~y_i\not=0\iff i\in I\}$ is non-empty.

We turn back to our setting. The $\mathrm{Cl}(X)$-grading of the Cox ring $\mathcal R(X)=\mathcal O(\overline X)$ yields a diagonal action of the algebraic torus $H:=(\KK^*)^n=\Spec(\KK[\mathrm{Cl}(X)])$ on $\overline X=(\KK^n)^2$ where
\[
 h\cdot (x,y)\ =\ (h_1x_1\:,\ldots ,\:h_nx_n\:,\:h_1y_1\:,\ldots ,\:h_ny_n).
\]
Since the subalgebra $\mathcal O(\overline X)^{\mathbb G_a}$ inherits the $\mathrm{Cl}(X)$-grading, the $H$-action descends to its spectrum $\overline Y\subseteq\bigwedge^2\KK^{n+1}$, turning $\kappa$ into an equivariant morphism. Here the action is explicitly described by
\[h\cdot e_0\wedge e_j=h_j\;e_0\wedge e_j,\qquad h\cdot e_i\wedge e_j=h_ih_j\; e_i\wedge e_j.\]
Note that this action differs from the well known maximal torus action. It rather is a submaximal action, with some connection to the maximal one, see Proposition~\ref{pro:git_subfan}.

In order to obtain the GIT-fan $\Lambda_H(\overline Y)$ we consider the {\it two-block partitions} of $N:=\{1,\ldots,n\}$, i.e. partitions where $N$ is a union of two disjoint subsets $A,A^c$. To each such partition $R=\{A,A^c\}$ we associate the hyperplane 
\[
 \mathcal H_R:=\left\{x\in\QQ^n;~\sum_{i\in A}x_i=\sum_{i\in A^c}x_i\right\}.
\]

\begin{theorem}
\label{thm:fan}
Consider the above $H$-action on the affine cone $\overline Y$ over the Grassmann variety $\mathrm{Gr}(2,n+1)$ and set $\Omega:=\QQ_{\ge 0}^n$. The GIT-fan $\Lambda_H(\overline Y)$ is the fan supported on $\Omega$ with walls given by the intersections $\mathcal H_R\cap \Omega$ where $R$ runs through the two-block partitions of $N$.
\end{theorem}


The key step of the proof is relate our submaximal $H$-action on $\overline Y$ to the maximal torus action on the smaller Grassmannian cone $\overline Y^\star$, see Proposition~\ref{pro:git_subfan}. The latter action is well understood, in particular the GIT-fan was described in \cite[Example~3.3.21]{dolgachevhu:GIT}.

The first step, however, is to provide a description of the $\overline Y$- and $\overline Y^\star$-sets. We need some further notation:
\[N:=\{1,\ldots,n\}\qquad \mathbf N:=\{\{i,j\};\ 1\le i<j\le n\}\]
\[N_{0}:=\{0,\ldots,n\}\qquad \mathbf N_0:=\{\{i,j\};~0\le i<j\le n\}\]
Recall that the cones over the Grassmannians lie in the wedge products $\overline Y^\star\subseteq\bigwedge^2\KK^n$, $\overline Y\subseteq\bigwedge^2\KK^{n+1}$. We use the above index sets $\mathbf N$ and $\mathbf N_0$ to refer to the coordinate indices where $\{i,j\}$ labels $e_i\wedge e_j$.

\begin{proposition}
\label{pro:ffaces}
A subset $I\subseteq \mathbf N_0$ is a $\overline Y$-set if and only if $I$ satisfies the following condition
 \[(*)\qquad\{i,j\},\{k,l\}\in I
 \quad\Longrightarrow\quad
 \{j,l\},\,\{i,k\}\in I
 \quad\text{or}\quad
 \{j,k\},\,\{i,l\} \in I.\]
 \end{proposition}
 \begin{proof}
 It follows from the nature of the Pl\"ucker relations that a $\overline Y$-set $I$ has in fact the property $(*)$. We prove that a subset of $N$ satisfying $(*)$ is a $\overline Y$-set by induction on $n$. For this recall that we have commutative diagram of closed embeddings
 \[
  \xymatrix{
 {C(\mathrm{Gr}(2,n+1))}
 \ar@{=}[r]
 &
  {\overline Y}
  \ar[r]
  &
  {\left(\KK e_0\bigwedge\KK^n\right)\oplus\bigwedge\nolimits^2\KK^n}
  \ar@{=}[r]
  &
  {\bigwedge\nolimits^2\KK^{n+1}}
  \\
  {C(\mathrm{Gr}(2,n))}
 \ar@{=}[r]
 \ar[u]
 &
  {\overline Y^\star}
  \ar[r]
  \ar[u]
  &
  {\bigwedge\nolimits^2\KK^n}
  \ar[u]
  }
 \]
where the embedding of the surrounding wedge products is reflected by the inclusion $\mathbf N\subseteq\mathbf N_0$. Let $I\subseteq\mathbf N_0$ be a set with the property $(*)$. If $I\subseteq \mathbf N$ holds, then the assertion follows from the induction hypothesis. We turn to the case where there exists $k\in N$ such that $\{0,k\}$ lies in $I$. We will explicitly construct an element $z\in\overline Y$ for which $z_{ij}$ vanishes if and only if $\{i,j\}$ does not lie in $I$. For this we introduce two graph graph structures on $N$ by $\mathcal G_{12}:=(N,\mathcal E_1\cup\mathcal E_2)$ and $\mathcal G_2:=(N,\mathcal E_2)$, where $\mathcal E_1$, $\mathcal E_2$ are sets of edges on $N$ defined by 
  \begin{align*}
   \mathcal E_1\ &:=\ \Big\{\{i,j\}\in I;~\{0,i\}\in I\text{ or }\{0,j\}\in I\Big\},\\
   \mathcal E_2\ &:=\ \Big\{\{i,j\}\in\mathbf N_0\setminus I;~\{0,i\},\{0,j\}\in I\Big\}.
  \end{align*}
From the definition of the edge sets of the respective graphs we know that if $\{i\}$ is a connected component of $\mathcal G_{12}$, then it also is a connected component of $\mathcal G_2$. Let $\mathcal F_1,\ldots,\mathcal F_q$ be the connected components of $\mathcal G_2$ different from a component $\{i\}$ of $\mathcal G_{12}$. We define a vector $x\in\KK^n$ by
\[
 x_i\ :=\ \begin{cases}
       0 & \text{if }\{i\}\text{ is a component of }\mathcal G_{12},\\
       p & \text{if }\{i\}\subseteq\mathcal F_p \text{ holds}.
      \end{cases}
\]


Moreover, we define $y\in \KK^n$ by $y_j:=1$ if $\{0,j\}\in I$ and $y_j:=0$ if $\{0,j\}\notin I$. We then claim that $z:=(1,x)\wedge (0,y)$ has the property
\[z_{ij}\not=0\quad\iff\quad\{i,j\}\in I.\]

Since $z_{0j}=y_j$ holds, it is clear that the claim is true for the components of this type. For $0\not=i<j$ the components of $z$ can be written as
\[
 z_{ij}\ =\ 
 x_i\,y_j\,-\,x_j\,y_i\ =\ 
 \begin{cases}
 0 & \text{if }\{0,i\},\{0,j\}\notin I,\\
 \pm\,x_i & \text{if }\{0,i\}\notin I,\{0,j\}\in I,\\
 x_i-x_j & \text{if }\{0,i\},\{0,j\}\in I.
 \end{cases}
\]
We now go through these three cases and verify for each that $\{i,j\}$ lies in $I$ if and only if $z_{ij}\not=0$ holds.

Assume that $\{0,i\},\{0,j\}\notin I$ holds and recall that there exists a $k\in N$ with $\{0,k\}\in I$. It follows from $(*)$ applied to $\{0,k\},\,\{i,j\}$ that $\{i,j\}$ does not lie in $I$.

For the second case suppose that $\{0,i\}\notin I$ and $\{0,j\}\in I$ hold. We then have
 \begin{align*}
  x_i\not=0 &\iff \text{there exists }l\in N\text{ such that }\{i,l\}\in\mathcal E_1\text{ or }\{i,l\}\in\mathcal E_2\\
  & \iff \text{there exists }l\in N\text{ such that }\{i,l\}\in\mathcal E_1\\
  & \iff \{i,j\}\in I
 \end{align*}
For the second equivalence note that $\{0,i\}\notin I$ holds which implies $\{i,l\}\notin\mathcal E_2$. The third equivalence is due to an application of $(*)$ to $\{0,j\}\,\,\{i,l\}$.

In the last case where $\{0,i\},\{0,j\}\in I$ holds we obtain
\begin{align*}
  x_i=x_j & \iff 
  \begin{array}[t]{ll}
  i,j\text{ lie in the same connected component of }\mathcal G_2\\
  \text{or}\quad\{i\},\{j\} \text{ are connected components of }\mathcal G_{12}
  \end{array}
  \\
  &\iff 
     \{i,j\}\in\mathcal E_2 \text{ or }\{i\},\{j\} \text{ are connected components of }\mathcal G_{12}\\
  &\iff \{i,j\}\notin I
 \end{align*}
For the second equivalence we use that each connected component of $\mathcal G_2$ is a complete graph, which follows $(*)$.
\end{proof}

 \begin{remark}
 The affine cone $\overline Y^\star$ over the smaller Grassmannian $\mathrm{Gr}(2,n)$ is invariant under the $H$-action. The corresponding GIT-fan $\Lambda_H(\overline Y^\star)$ of this restricted action is well known, it was described in terms of walls in \cite[Example~3.3.21]{dolgachevhu:GIT} and \cite[Example~8.5]{gitviacox} as follows: Set 
 \[\Omega^\star:=\mathrm{cone}(e_i+e_j;~1\le i<j\le n)\subseteq\QQ_{\ge 0}^n.\]
 Then the GIT fan $\Lambda_H(\overline Y^\star)$ is the fan supported on $\Omega^\star$ with walls given by the intersections of $\Omega^\star$ with the above hyperplanes $\mathcal H_R$.
\end{remark}

\begin{proposition}
\label{pro:git_subfan}
 The GIT-fan $\Lambda_H(\overline Y^\star)$ is a subfan of $\Lambda_H(\overline Y)$.
\end{proposition}

\begin{remark}
 Proposition~\ref{pro:git_subfan}, the universal property of the limit quotient (see \cite[Remark 2.3]{on_chow_quotients}) and the inclusion $\overline Y^\star\subseteq\overline Y$ show that there is a closed embedding of the respective limit quotients $\overline M_{0,n}=\overline Y^\star\limquot H\subseteq \overline Y\limquot H$. 
 
 Moreover, for any $\lambda\in\Lambda_H(\overline Y^\star)$ the preimage $p^{-1}({\overline Y^\star}^\mathrm{ss}(\lambda))$ lies in ${\overline Y}^\mathrm{ss}(\lambda)$ where 
 \[
  p\colon \bigwedge\nolimits^2\KK^{n+1}\ =\ \KK e_0\bigwedge\KK^n\ \oplus\ \bigwedge\nolimits^2\KK^n\quad \to\quad \bigwedge\nolimits^2\KK^n
 \]
is the projection. This gives rise to an open subset $U\subseteq\overline Y\limquot H$ of the limit quotient and a surjective morphism $\overline p\colon U\to\overline M_{0,n}$.
\end{remark}

\begin{example}
Consider the weights of the coordinates of the $H$-action on $\bigwedge^2\KK^{n+1}$
\[
w_{01}:=e_1,\quad\ldots,\quad w_{0n}:=e_n,
\qquad
w_{jk}:=e_j+e_k,\ 1\le j<k\le n.
\]
 The following pictures of polytopal complexes arise from intersecting the GIT-fan $\Lambda_H(\overline Y)$ with the hyperplane given by $1=x_1+\ldots+x_n$ in the cases $n=3,4$. The shaded area indicates the support $\Omega^\star$ of $\Lambda_H(\overline Y^\star)$.
 
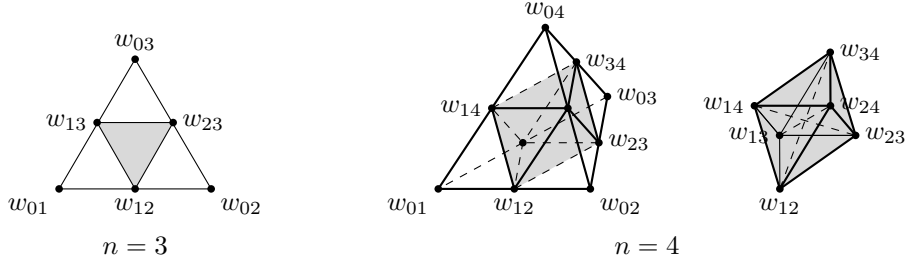
 \[
 \begin{array}{ccc}
 \gitfani
 &
 \qquad
 &
 \gitfanii
 \quad
 \gitfaniii
 \\
 n=3
 &
 \qquad
 &
 n=4 
 \end{array}
\]
In the case $n=3$ the three walls of the GIT-fan are generated by two of the vectors $w_{12},w_{13},w_{23}$ and correspond to the two-block partitions 
\[
 \{\{1\},\{2,3\}\},\qquad
 \{\{2\},\{1,3\}\}
 \qquad\text{and}\qquad
 \{\{3\},\{1,2\}\}.
\]
In the case $n=4$ again the hyperplanes separating $\Omega^\star$ from the remaining $4$ cones correspond to the partitions of the type $\{\{i\},\{j,k,l\}\}$. The dotted lines in the right picture indicate the fan structure inside $\Lambda_H(\overline Y^\star)$. There are eight maximal cones arising from $3$ hyperplanes of the form $\{\{i,j\},\{k,l\}\}$.

\end{example}
 
\begin{proof}[Proof of Proposition~\ref{pro:git_subfan}] 
Recall that the weights of the coordinates of the $H$-action are
\[
w_{01}:=e_1,\quad\ldots,\quad w_{0n}:=e_n,
\qquad
w_{jk}:=e_j+e_k,\ 1\le j<k\le n.
\]
The GIT-fans $\Lambda_H(\overline Y)$ and $\Lambda_H(\overline Y^\star)$ are the collections of cones which arise as intersections of cones $\omega_I=\mathrm{cone}(w_{ij};~\{i,j\}\in I)$ associated to $\overline Y$- or $\overline Y^\star$-sets respectively. From Proposition~\ref{pro:ffaces} we know that every $\overline Y^\star$-set is also a $\overline Y$-set. This means we only have to show that for every $\overline Y$-set $I\subseteq\mathbf N_0$ there exists a $\overline Y^\star$-set $J\subseteq\mathbf N$ such that $\omega_I\cap\Omega^\star=\omega_J$ holds. For a $\overline Y$-set $I\subseteq\mathbf N_0$ we set
\[
 J:=J_1\cup J_2,\qquad J_1:=I\cap\mathbf N,\qquad J_2:=\{\{i,j\};~\{0,i\},\{0,j\}\in I\}
\]
and prove that $J$ has the required properties. We first claim that $J$ is an $\overline Y^\star$-set. For this we check that the condition of Proposition~\ref{pro:ffaces} applies to any two elements of $J$. If these two elements lie either both in $J_1$ or $J_2$ then the claim follows from $I$ being a $\overline Y$-set or the construction of $J_2$ respectively. For the remaining case consider $\{j,k\}\in J_1$ and $\{i_1,i_2\}\in J_2$. Since both $\{0,i_1\}$ and $\{j,k\}$ lie in $I$, we can without loss of generality assume that also $\{i_1,j\}$ and $\{0,k\}$ lie in $I$. Finally with $\{0,i_2\}\in I$ we conclude that $\{i_2,j\},\{i_1,k\}$ are elements of $J$. This shows that $J$ is a $\overline Y^\star$-set.

We now prove $\omega_I\cap\Omega^\star=\omega_J$. It is easy to see that $\omega_J$ is in fact contained in $\omega_I\cap\Omega^\star$; we turn to the reverse inclusion. With non-negative $a_i,a_{jk}$ let 
\[x\ :=\ \sum_{I\setminus \mathbf N}a_iw_{0i}\ +\ \sum_{I\cap\mathbf N} a_{jk}w_{jk}\]
lie in $\omega_I\cap\Omega^\star$. We show that $x$ is a non-negative linear combination of elements $w_\eta,\ \eta\in J$.
Let $a_{i_1}$ be minimal among all $a_i$ with $\{0,i\}\in I$. For an arbitrary $\{0,i_2\}\in I$ we then replace in the above sum
\[a_{i_1}w_{0i_1}\ +\ a_{i_2}w_{0i_2}\qquad\text{by}\qquad (a_{i_2}-a_{i_1})w_{0i_2}\ +\ a_{i_1}w_{i_1i_2} .\]
Note that now $\{i_1,i_2\}$ lies in $J_2$. Iterating this process we see that there exists some $\{0,i\}\in I$ such that $x$ has the form
\[
(**)\qquad\qquad
x\ =\ b_iw_{0i}\ +\ \sum_{J_1\cup J_2}b_{jk}w_{jk}.
\]
Without loss of generality we assume that $i=1$ holds. The condition $x\in\Omega^\star$ implies $x_1\le x_2+\ldots+x_n$, hence we have
\[
b_{1}\ \le\ 
2\!\!\!\!\sum_{
\scalebox{0.6}{$
\begin{array}{cc}
 \{j,k\}\in J\\
 j,k\not= 1
\end{array}
$}
}\!\!\!\!b_{jk}
\qquad
\text{and}
\qquad
b_{1}\ =\ 
2\!\!\!\!\sum_{
\scalebox{0.6}{$
\begin{array}{cc}
 \{j,k\}\in J\\
 j,k\not= 1
\end{array}
$}
}\!\!\!\!b'_{jk}
\]
for certain $0\le b'_{jk}\le b_{jk}$. Plugging $w_{01}=\;^1\!/_2(w_{1j}+w_{1k}-w_{jk})$ into $(**)$ we obtain a non-negative linear combination
\[
 x\ =\ \!\!\!\!\sum_{
 \scalebox{0.6}{$
\begin{array}{cc}
 \{j,k\}\in J\\
 j,k\not= 1
\end{array}
$}
}\!\!\!\!
\left((b_{1j}+b_{jk}')w_{1j}+(b_{1k}+b_{jk}')w_{1k}+(b_{jk}-b_{jk}')w_{jk}\right)
+
\sum_{\scalebox{0.6}{$\{1,k\}\in J$}}\!\!b_{1k}w_{1k}.
\]
The last step to show is that for $\{j,k\}\in J$ both $\{1,j\}$ and $\{1,k\}$ lie in $J$. Recall that we have $\{0,1\}\in I$. If $\{j,k\}$ lies in $J_2$, then this follows directly from construction of $J_2$. Otherwise we can without loss of generality assume that $\{0,j\},\{1,k\}$ lie in $I$. The claim again follows from the construction of $J_2$.
\end{proof}

\begin{proof}[Proof of Theorem~\ref{thm:fan}]
As before we denote the weights of coordinates with respect to the $H$-action by $w_{0i}=e_i,\ w_{jk}=e_j+e_k$. From Proposition~\ref{pro:git_subfan} we know that $\Lambda_H(\overline Y)$ has the asserted form on $\Omega^\star$. Note that the remaining support $\Omega\setminus\mathrm{relint}(\Omega^\star)$ is the union of the cones
\[
 \sigma_i\ :=\ \mathrm{cone}(w_{ij};~j\in N\setminus\{i\}),\qquad i=1,\ldots, n.
\]
None of the hyperplanes $\mathcal H_R$ intersect $\sigma_i$ in its relative interior. This means that we have to prove that $\sigma_i$ is a cone in the GIT-fan $\Lambda_H(\overline Y)$, i.e. the intersection of cones $\omega_I$ associated to $\overline Y$-sets. Note that $\sigma_i$ itself is a cone associated to a $\overline Y$-set. Hence, it suffices to show that for any $\overline Y$-set $I\subseteq \mathbf N_0$ the intersection $\omega_I\cap \sigma_i$ is a face of $\sigma_i$. Without loss of generality we assume that $i$ equals $1$ and set $\sigma:=\sigma_1$. We now claim that $\omega_I\cap\sigma=\omega_J$ holds where
\[
J:=J_1 \cup J_2;\qquad J_1:=I\cap\{\{1,j\};~j\in N_0\setminus\{1\}\},\quad J_2:=\{\{1,j\};~\{0,j\}\in I\}.
\]
To prove $\omega_J\subseteq\omega_I\cap\sigma$ note that any $w_{1j}$ with $\{1,j\}\in J_1$ clearly lies in $\omega_I\cap\sigma$. Hence, it suffices to show that for $w_{1j}$ with $\{0,j\}\in I$ the same holds. In case $\{0,1\}\in I$ this follows from $w_{1j}=w_{01}+w_{0j}\in \omega_I\cap\sigma$. Otherwise there must exist $\{1,l\}\in I$ and from Proposition~\ref{pro:ffaces} we know $\{0,l\},\{1,j\}\in I$. This implies $w_{1j}\in\omega_I\cap\sigma$.

For the reverse inclusion $\omega_I\cap\sigma\subseteq\omega_J$ consider the non-negative linear combination
\[
 x\ :=\ 
 a_{01}w_{01}\ +
 \!\!\!\!
 \sum_{
 \scalebox{0.6}{$
\begin{array}{cc}
 \{1,j\}\in I\\
 j\not= 0
\end{array}
$}
}\!\!\!a_{1j}w_{1j}\ +
\!\!\!\!
 \sum_{
 \scalebox{0.6}{$
\begin{array}{cc}
 \{0,j\}\in I\\
 j\not= 1
\end{array}
$}
}\!\!\!a_{0j}w_{0j}\ +
\!\!\!\!
 \sum_{
 \scalebox{0.6}{$
\begin{array}{cc}
 \{j,k\}\in I\\
 j,k\not= 0,1
\end{array}
$}
}\!\!\!\!\!a_{jk}w_{jk}\ \in\ \omega_I
\]
Since $x$ lies in $\sigma$, we have $x_1\ge x_2+\ldots+x_n$ and this amounts to
\[
 a_{01}\ \ge
\!\!\!\!
 \sum_{
 \scalebox{0.6}{$
\begin{array}{cc}
 \{0,j\}\in I\\
 j\not= 1
\end{array}
$}
}\!\!\!a_{0j}\ +
2
\!\!\!\!
 \sum_{
 \scalebox{0.6}{$
\begin{array}{cc}
 \{j,k\}\in I\\
 j,k\not= 0,1
\end{array}
$}
}\!\!\!\!\!a_{jk}.
\]
If $\{0,1\}\notin I$ holds, i.e. $a_{01}=0$, then $x$ lies in the cone generated by the $w_{1j}$, $\{1,j\}\in J_1$. Otherwise with
$w_{0j}=w_{1j}-w_{01}$ and $w_{jk}=w_{1j}+w_{1k}-2w_{01}$ we get a non-negative linear combination
\begin{align*}
 x&=
 \!\!\!\!
 \sum_{
 \scalebox{0.6}{$
\begin{array}{cc}
 \{1,j\}\in I\\
 j\not= 0
\end{array}
$}
}\!\!\!a_{1j}w_{1j}\ +
\!\!\!\!
 \sum_{
 \scalebox{0.6}{$
\begin{array}{cc}
 \{0,j\}\in I\\
 j\not= 1
\end{array}
$}
}\!\!\!a_{0j}w_{1j}\ +
\!\!\!\!
 \sum_{
 \scalebox{0.6}{$ 
\begin{array}{cc}
 \{j,k\}\in I\\
 j,k\not= 0,1
\end{array}
$}
}\!\!\!\!\!a_{jk}(w_{1j}+w_{1k})\\
&~
\quad
+
\left(
a_{01}\ -\!\!\!\!
 \sum_{
 \scalebox{0.6}{$
\begin{array}{cc}
 \{0,j\}\in I\\
 j\not= 1
\end{array}
$}
}\!\!\!a_{0j}\ -
2
\!\!\!\!
 \sum_{
 \scalebox{0.6}{$
\begin{array}{cc}
 \{j,k\}\in I\\
 j,k\not= 0,1
\end{array}
$}
}\!\!\!\!\!a_{jk}
\right)
w_{01}.
\end{align*}
The last thing to check is that all the above $w_{ij}$ lie in $\omega_J$. For this suppose that $\{j,k\}\in I$ holds. Since $\{0,1\}$ is contained in $I$, it follows from the construction of $J$ that both $\{1,j\}$ and $\{1,k\}$ lie in $J$.
\end{proof}

\section{Combinatorial blow-ups}
\label{sec:combinatorics}

In this section we will provide a criterion whether a given cone lies in the iterated stellar subdivision of a simplicial fan. In \cite{incidence_combinatorics} Feichtner and Kozlov deal with this problem in the more general setting of semilattices and give a nice characterisation in the case where the collection of subdivided cones forms a building set. We approach the issue of blowing up non-building sets, see Theorem~\ref{thm_main_general}. For details on stellar subdivisions see e.g. \cite[Definition 5.1]{cox_and_combi_2}.

Let $\mathcal V$ be a family of rays in a vector space and consider a $\mathcal V$-fan $\Sigma_0$, i.e. a fan with rays given by $\mathcal V$. We then choose additional rays $\nu_i,~i=1,\ldots,r$ lying in the relative interiors $\sigma_i^\circ$ of pairwise different cones $\sigma_i\in\Sigma_0$. Moreover, we assume that $\sigma_i\precneqq \sigma_j$ implies $j<i$, which means that the larger the cone the earlier it will be subdivided. Now the questions comes up what the cones of the fan $\Sigma_r$ are which arises from $\Sigma_0$ by the subsequent stellar subdivisions in the rays $\nu_i$.

We call a subset $\mathcal S'$ of $\mathcal S:=\{\sigma_1,\ldots,\sigma_r\}$ {\it conjunct}, if the union $\bigcup_{\sigma\in\mathcal S'}(\sigma\setminus\{0\})$ is a connected subset in the usual sense and we set
\[
 \langle\mathcal S\rangle:=\left\{\sum_{\sigma\in \mathcal S'}\sigma\,;~\mathcal S'\subseteq\mathcal S\text{ conjunct}  \right\}.
\]
A collection $\mathcal C\subseteq\mathcal V\cup\mathcal S$ is called {\it geometrically nested}, if for any subset $\mathcal H\subseteq \mathcal C$ of pairwise incomparable cones with $|\mathcal H|\ge 2$ the following holds:
 \[
  \sum_{\tau\in\mathcal H} \tau\;\in\;\Sigma_0\setminus\langle\mathcal S\rangle.
 \]

\begin{proposition}
\label{pro_main_fan}
Let $\Sigma_0$ be a simplicial fan and $\nu_i\in\sigma_i^\circ$ rays in the relative interiors of pairwise different cones $\sigma_i\in\Sigma_0$. Assume that $\sigma_i\precneqq \sigma_j$ implies $j<i$ and let $\Sigma_r$ be the iterated stellar subdivision of $\Sigma_0$ in the rays $\nu_1,\ldots,\nu_r$ in order of ascending indices. If in the above notation $\mathcal C\subseteq \mathcal V\cup\mathcal S$ is geometrically nested, then $\mathrm{cone}(v,\nu_i;\ v\in\mathcal C\cap\mathcal V,\ \sigma_i\in\mathcal C\cap \mathcal S)$ lies in $\Sigma_r$.
\end{proposition}

We will prove this using the technique of combinatorially blowing up elements in a semilattice developed by Feichtner and Kozlov in \cite{incidence_combinatorics}. For this we introduce some notation. Let $(\mathcal L,\le)$ be a finite {\it (meet)-semilattice}, i.e. a finite partially ordered set such that any non-empty subset $\mathcal X\subseteq\mathcal L$ posesses a greatest lower bound $\bigwedge \mathcal X$ called {\it meet}. Any meet-semilattice has a unique minimal element $0$. Moreover, for a subset $\mathcal X\subseteq \mathcal L$ the set 
$\{z\in\mathcal L;~z\ge x\text{ for all }x\in \mathcal X\}$
is either empty or has a unique minimal element $\bigvee \mathcal X$ called {\it join}. For $y\in\mathcal L$ we denote $\mathcal X_{\le y}:=\{x\in \mathcal X;~x\le y\}$, and finally, the semilattice $\mathcal L$ is called {\it distributive} if the equation $x\wedge(y\vee z)=(x\wedge y)\vee (x\wedge z)$ holds for any $x,y,z\in\mathcal L$.

We now turn to blow-ups of semilattices in the sense of \cite[Definition~3.1]{incidence_combinatorics}. The {\it blow-up} of $(\mathcal L,\ge)$ in an element $\xi\in\mathcal L$ is the 
semilattice $\mathrm{Bl}_{(\xi)}(\mathcal L)$ consisting of
\[
 x\in\mathcal L\text{ with }x\not\ge\xi
 \qquad\text{and}\qquad
 (\xi,x)\text{ where }\mathcal L\ni x\not\ge \xi\text{ and }x\vee \xi\text{ exists}.
\]
The order relation $>_{\mathrm{Bl}}$ of the blow-up is given by
\[
 x>_{\mathrm{Bl}}y\;\;\text{ if }\;\;x>y,\qquad
 (\xi,x)>_{\mathrm{Bl}}(\xi,y)\;\;\text{ if }\;\;x>y,\qquad
 (\xi,x)>_{\mathrm{Bl}} y\;\;\text{ if }\;\;x\ge y,
\]
where in all three cases $x,y\not\ge \xi$ holds.

We now want to iterate this process. Let $\mathcal G=(\xi_1,\ldots, \xi_r)$ be a family of elements $\xi_i\in\mathcal L$. The blow-up of $\mathcal L$ in $\mathcal G$ is simply the subsequent blow-up of $\mathcal L$ in the elements $\xi_i$ in order of ascending indices. When we speak of a subfamily $(\xi_{i_1},\ldots \xi_{i_s})$ of $\mathcal G$ we always tacitly assume, that the order is preserved, i.e. that $j<k$ implies $i_j<i_k$. We call $\mathcal G$  {\it sorted} if $\xi_i>\xi_j$ implies $i<j$. Moreover, we denote the underlying set of the family $\mathcal G$ by $\mathcal S_\mathcal G$.

The subset $\mathcal S\subseteq\mathcal L\setminus\{0\}$ is a {\it building set} for $\mathcal L$, if for every $x\in\mathcal L\setminus \{0\}$ the interval $\{y\in\mathcal L;\ 0\le y\le x\}$ as a poset decomposes into a product of certain smaller intervals given by elements of $\mathcal S$. For a precise definition see \cite[Definition~2.2]{incidence_combinatorics}. For two subsets $\mathcal C\subseteq\mathcal S\subseteq\mathcal L$ we call $\mathcal C$ {\it nested} (in $\mathcal S$) if for any subset $\mathcal H\subseteq \mathcal C$ of pairwise incomparable elements with $|\mathcal H|\ge 2$ the join $\bigvee\mathcal H$ exists and does not lie in $\mathcal S$. Note that the collection of nested sets forms an abstract simplicial complex $\mathfrak C(\mathcal S)$ with vertex set $\mathcal S$.
\begin{theorem}[{\cite[Theorem~3.4]{incidence_combinatorics}}]
\label{thm:feichtner_kozlov}
Assume that $\mathcal G$ is a sorted familiy in the semilattice $\mathcal L$ such that the underlying set $\mathcal S_\mathcal G$ is a building set. Then we have an isomorphism of posets
\[
 \mathfrak C(\mathcal S_\mathcal G)\to\mathrm{Bl}_{\mathcal G}(\mathcal L);
 \qquad
 \mathcal C\mapsto \bigvee_{\xi\in \mathcal C}(\xi,0).
 \]
\end{theorem}

We now describe a suffient criterion to test whether an element lies in $\mathrm{Bl}_\mathcal F(\mathcal L)$ in the case where $\mathcal S_\mathcal F$ is not a building set.

\begin{theorem}
\label{thm_main_general}
Let $\mathcal F$ be a sorted family in $\mathcal L$ and consider a subset $\mathcal C$ of the underlying set $\mathcal S_\mathcal F$. If there exists a building set $\mathcal S$ of $\mathcal L$ with $\mathcal S_\mathcal F\subseteq\mathcal S$ such that $\mathcal C$ is nested in $\mathcal S$, then $\bigvee_{\xi\in\mathcal C}(\xi,0)$ exists in the blow-up $\mathrm{Bl}_{\mathcal F}(\mathcal L)$.
\end{theorem}

Before we enter the proof of the Theorem we consider an example. Furthermore, for distributive $\mathcal L$ we provide an explicit construction of such a building set in the case where $\mathcal S_\mathcal F$ generates $\mathcal L$ by $\vee$, see Construction~\ref{cst:building_from_set}, Lemma~\ref{lem_buildset_from_harmonious}.

\begin{example}
The face poset of a polyhedral fan is a semilattice in which the stellar subdivision in a ray $\nu\in\sigma^\circ$ corresponds to the blow-up of the element $\sigma$, see \cite[Proposition 4.9]{incidence_combinatorics}. Viewing the positive orthant $\Sigma:=\QQ^3_{\ge 0}$ as a fan, we ask for the combinatoric structure of its stellar subdivisions $\Sigma_1$ and $\Sigma_2$ in the sorted families
\[
 \mathcal G_1\;:=\;(\nu_1,\nu_2,e_1,e_2,e_3),\qquad 
 \mathcal G_2\;:=\;(\nu_2,\nu_1,e_1,e_2,e_3),
 \]
 \[
 \text{where}
 \quad
 \nu_1\;:=\;(1,1,0),\quad
 \nu_2\;:=\;(0,1,1). 
\]
If $\mathcal G_1$ and $\mathcal G_2$ were building sets, then Theorem~\ref{thm:feichtner_kozlov} would imply that the fans $\Sigma_1$ and $\Sigma_2$ coincide. Clearly, this is not the case.
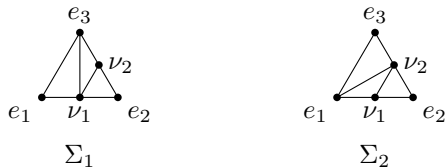
\[
\begin{array}{cc}
 \blowupi
 \qquad
 &
 \qquad
 \blowupii
 \\
 {\Sigma_1} \qquad&\qquad {\Sigma_2} 
\end{array}
\]
We now add to $\mathcal G_1$ and $\mathcal G_2$ a ray lying in the relative interior of the join of the faces $\mathrm{cone}(e_1,e_2)$ and $\mathrm{cone}(e_2,e_3)$, e.g. $\nu_0=(1,1,1)$. This yields two building sets
\[
 \mathcal G_{1a}\;:=\;(\nu_0,\nu_1,\nu_2,e_1,e_2,e_3),\qquad 
 \mathcal G_{2a}\;:=\;(\nu_0,\nu_2,\nu_1,e_1,e_2,e_3).
 \]
Both families give rise to the same subdivided fan. Note that the faces of $\Sigma_{1a}=\Sigma_{2a}$ not having $\nu_0$ as a ray lie in both $\Sigma_1$ and $\Sigma_2$. This is essentially the idea of the proof of Proposition~\ref{pro_main_fan}.
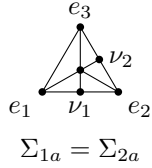
\[
\begin{array}{c}
 \blowupiii 
 \\
 {\Sigma_{1a}=\Sigma_{2a}}
\end{array}
\]
\end{example}


\begin{definition}
 Let $\mathcal S=\{\xi_1,\ldots,\xi_r\}$ be a subset of the semilattice $\mathcal L$. We call a (non-ordered) pair $\{\xi_i,\xi_j\}$ {\it harmonious} (with respect to $\mathcal S$) if at least one of the following conditions is satisfied:
 \[
   \xi_i\wedge\xi_j=0
   \qquad\text{or}\qquad
   \xi_i\vee\xi_j\text{ does not exist}
   \qquad\text{or}\qquad
   \xi_i\vee\xi_j\in \mathcal S.
 \]
\end{definition}

\begin{construction}
\label{cst:building_from_set}
 Let $\mathcal S=\{\xi_1,\ldots,\xi_r\}$ be a subset of $\mathcal L$. For all pairs of non-harmonious elements $\{\xi_i,\xi_j\}$ we add to $\mathcal S$ the element $\xi_i\vee\xi_j$:
 \[
  \mathcal S'\quad:=\quad\mathcal S\ \cup\ \{\xi_i\vee\xi_j;\ \{\xi_i,\xi_j\}\text{ non-harmonious with respect to }\mathcal S\}.
 \]
 We continue this process with the new set $\mathcal S'$ instead of $\mathcal S$ until all pairs are harmonious and denote the final set by $\langle\!\langle \mathcal S\rangle\!\rangle$. Since $\mathcal L$ is finite, clearly this process terminates after finitely many steps. 
\end{construction}

\begin{lemma}
\label{lem_buildset_from_harmonious}
Assume that $\mathcal L$ is distributive and a subset $\mathcal S\subseteq\mathcal L\setminus\{0\}$ generates it by $\vee$. Then the following assertions hold.
\begin{enumerate}
 \item If for any $x\in\mathcal L$ and distinct $\xi_i,\xi_j\in\max (\mathcal S_{\le x})$ their meet $\xi_i\wedge \xi_j$ equals $0$, then $\mathcal S$ is a building set for $\mathcal L$.
 \item The set $\langle\!\langle\mathcal S\rangle\!\rangle$ is a building set for $\mathcal L$.
 \end{enumerate}
\end{lemma}

\begin{proof}
For the proof of (i) we check the two conditions of \cite[Proposition~2.3~(4)]{incidence_combinatorics}. Fix an $x\in\mathcal L$ and a subset $\{y,y_1,\ldots,y_t\}\subseteq\mathrm{max}\,(\mathcal S_{\le x})$. By assumption we have
  \[
   0=(y\wedge y_1)\vee\ldots\vee(y\wedge y_t)=y\wedge(y_1\vee\ldots\vee y_t). 
  \]
Since $0\not\in \mathcal S$ holds, this implies $\mathcal S_{\le y}\,\cap\,\mathcal S_{\le y_1\vee\ldots\vee y_t}=\emptyset$. For the second condition let $z<y$. Clearly $z\vee y_1\vee\ldots \vee y_t\le y\vee y_1\vee\ldots \vee y_t$ holds. If they were equal, so would be the respective meets with $y$ and this would imply $z=y$.
  
We now prove the second assertion (ii). By construction of $\langle\!\langle\mathcal S\rangle\!\rangle$, for any $x\in \mathcal L$ and $\xi_i,\xi_j\in\max\,(\langle\!\langle\mathcal S\rangle\!\rangle_{\le x})$ the pair $\{\xi_i,\xi_j\}$ is harmonious (with respect to $\langle\!\langle\mathcal S\rangle\!\rangle$). Its join exists but - by maximality of $\xi_i$ and $\xi_j$ - does not lie in $\langle\!\langle\mathcal S\rangle\!\rangle$. This implies that $\xi_i\wedge\xi_j=0$ holds and the assertion follows from (i).
\end{proof}

\begin{proof}[Proof of Theorem~\ref{thm_main_general}]
Before we enter the proof let us recall the join rules of blow-ups from \cite[Lemma~3.2]{incidence_combinatorics}. Let $x,y,\xi$ lie in the semilattice $\mathcal L$ and consider the blow-up $\mathcal L'$ of $\mathcal L$ in $\xi$. Then the join $(\xi,x)\vee_{\mathcal L'} y$ exists if and only if $x\vee_\mathcal Ly$ exists and $x\vee y\not\ge \xi$ holds. The join $x\vee_{\mathcal L'} y$ exist if and only if $x\vee_\mathcal Ly$ exists. In case the joins exist the following formulae hold
\[
(\xi,x)\vee_{\mathcal L'} y\ =\ (\xi,x\vee_{\mathcal L} y),
\qquad\qquad
\quad x\vee_{\mathcal L'} y\ =\ x \vee_{\mathcal L} y.
\]
We turn back to our case and fix some notation. We write $\mathcal F=(\xi_1,\ldots,\xi_r)$ and denote the  elements lying in $\mathcal C$ by $\xi_{i_j},\ j=1,\ldots,s$ where we assume that the order is preserved, i.e. $j<j'$ is equivalent to $i_{j}<i_{j'}$. Moreover, for $k=1,\ldots,r$ let $\mathcal L_k$ be the blow-up of $\mathcal L$ in $(\xi_1,\ldots,\xi_k)$ and for consistency we set $\mathcal L_0:=\mathcal L$. In $\mathcal L_k$ we consider the following (a priori non-existent) join 
\[
 \bigvee_{j=1}^{j(k)}(\xi_{i_j},0)\ \vee\ \bigvee_{j=j(k)+1}^s\xi_{i_j},
 \quad
 \text{where}
 \quad
 j(k):=\max(\{0\}\cup\{j;\ i_j\le k\}).
\]
In case this join does exist, we denote it by $z_k$. Note that from the definition of $j(k)$ it follows that $i_{j(k)}$ is the largest index, such that $\xi_{i_1},\ldots,\xi_{i_{j(k)}}$ are among the $\xi_1,\ldots,\xi_k$. We prove the existence of $z_r=\bigvee_{\xi\in \mathcal C}(\xi,0)$ by induction on $k$. Since $\mathcal C$ is nested, it is clear that $z_0=\bigvee\mathcal C$ does exist in $\mathcal L_0$. Now assume that $z_k\in\mathcal L_k$ exists. We discriminate two possible cases: In the first case $\xi_{k+1}$ does not lie in $\mathcal C$ in the second case it does.

Assume that $\xi_{k+1}\notin\mathcal C$ holds and note that this is equivalent to $j(k)=j(k+1)$. Hence, as elements in $\mathcal L_k$ we have $z_{k+1}=z_k$ and the only thing to check is that $z_k\not\ge \xi_{k+1}$ holds. For this note that the iterated application of the above join rules shows that
\begin{align*}
 z_k&=(\xi_{i_{j(k)}},0)\vee\left(\bigvee_{j=1}^{j(k)-1}(\xi_{i_j},0)\vee\zeta_k\right)
 =\left(\xi_{i_{j(k)}},\bigvee_{j=1}^{j(k)-1}(\xi_{i_j},0)\vee\zeta_k\right)\\
 &=\ldots=\left(\xi_{i_{j(k)}},\left(\ldots \left(\xi_{i_1},\zeta_k\right)\ldots \right)\right)
 \qquad\text{where }\zeta_k:=\bigvee_{j=j(k)+1}^s\xi_{i_j}.
\end{align*}
If we had $z_k\ge \xi_{k+1}$, then this would mean $(\xi_{i_{j(k)-1}},(\ldots (\xi_{i_1},\zeta_k)\ldots ))\ge \xi_{k+1}$. Iterating this argument we would get $\zeta_k\ge \xi_{k+1}$ in $\mathcal L_0$ which would imply $\xi_{k+1}\in\mathcal S_{\le \zeta_k}$. Since $\mathcal S$ is a building set, by \cite[Proposition~2.8~(2)]{incidence_combinatorics}
\[
 \max(\mathcal S_{\le \zeta_k})=\max(\xi_{i_j},~j=j(k)+1,\ldots,s)
\]
holds. Hence there must exist $j_0\ge j(k)+1$ with $\xi_{k+1}\le\xi_{i_{j_0}}$. Since $\xi_{k+1}\notin\mathcal C$ holds, we have $\xi_{k+1}\not=\xi_{j_0}$. In particular, this implies $k>i_{j_0}-1\ge i_{j(k)+1}-1$. However, from the definition of $j(k)$ we easily see that $k\le i_{j(k)+1}-1$ holds, a contradiction.

We turn to the second case where $\xi_{k+1}\in\mathcal C$ holds which is equivalent to $j(k)+1=j(k+1)$. In $\mathcal L_k$ we consider the element
\[
  y_k\ :=\ \bigvee_{j=1}^{j(k)}(\xi_{i_j},0)\ \vee\ \bigvee_{j=j(k)+2}^s\xi_{i_j}.
\]
Since $z_k$ exists, it follows that also $y_k$ and the join $\xi_{k+1}\vee y_k$ exist. Then the last thing to show is that $y_k\not\ge \xi_{k+1}$ holds. This follows from the same argument as above with $y_k$ instead of $z_k$.
\end{proof}

\begin{proof}[Proof of Proposition~\ref{pro_main_fan}]
First note that since $\Sigma_0$ is simplicial so is the iterated stellar subdivision $\Sigma_r$. In particular, the further application of stellar subdivisions in the original rays $\mathcal V$ leaves $\Sigma_r$ unchanged. From \cite[Proposition~4.9]{incidence_combinatorics} we know that a stellar subdivision in a ray $\nu\in\sigma^\circ$ corresponds to the blow-up of the face poset of the original fan in $\sigma$. More precisely, as posets $\Sigma_r$ and $\mathrm{Bl}_{\mathcal F}(\Sigma_0)$ are isomorphic, where 
\[
\mathcal F:=(\sigma_1,\ldots,\sigma_r,v_1,\ldots, v_t),\qquad \mathcal V=\{v_1,\ldots,v_t\}.
\]
For the proof of the Proposition we now check the assumptions of Theorem~\ref{thm_main_general}. First note that $\Sigma_0$ is simplicial, hence it is distributive as a semilattice. Its joins and meets can be computed by taking convex geometric sums and intersections respectively. Also, with $\mathcal S=\{\sigma_1,\ldots,\sigma_r\}$ it is clear that $\mathcal V\cup\mathcal S$, the underlying set of $\mathcal F$, generates $\Sigma_0\setminus\{0\}$ by $+$. In particular, from Lemma~\ref{lem_buildset_from_harmonious} we infer that $\langle\!\langle \mathcal V\cup\mathcal S\rangle\!\rangle$ is a building set for $\Sigma_0$.
 
  Now note that $\langle\!\langle \mathcal V\cup\mathcal S\rangle\!\rangle\setminus \mathcal V$ equals $\langle\!\langle \mathcal S\rangle\!\rangle$ and from the respective constructions it follows that $\langle\!\langle \mathcal S\rangle\!\rangle\subseteq \langle \mathcal S\rangle$ holds. Together this means
 \[
  \Sigma_0\setminus\langle \mathcal S\rangle
  \ \subseteq\ 
  \Sigma_0\setminus \langle\!\langle \mathcal S\rangle\!\rangle
  \ =\ 
  \Sigma_0\setminus(\langle\!\langle \mathcal V\cup\mathcal S\rangle\!\rangle\setminus\mathcal V)
  \ =\ 
  (\Sigma_0\setminus\langle\!\langle \mathcal V\cup\mathcal S\rangle\!\rangle) \cup\mathcal V.
 \]
 Since $\mathcal C\subseteq\mathcal V\cup \mathcal S$ is geometrically nested, it follows that it is also nested in $\langle\!\langle \mathcal V\cup\mathcal S\rangle\!\rangle$ in the sense of semilattices. From Theorem~\ref{thm_main_general} we now know that $\bigvee_{c\in\mathcal C}(c,0)$ lies in $\mathrm{Bl}_\mathcal F(\Sigma_0)$. Under the above isomorphy $\Sigma_r\cong \mathrm{Bl}_\mathcal F(\Sigma_0)$ this means
 \[
  \mathrm{cone}(v,\nu_i;~v\in\mathcal C\cap \mathcal V,\,\sigma_i\in\mathcal C\cap\mathcal S)\in\Sigma_r.
 \]
\end{proof}

\section{The limit quotient as blow-up}
\label{sec:iterated_blow-up}
This section is devoted to the main result and its proof. As before, let $\overline{Y}\subseteq\bigwedge^2\KK^{n+1}$ be the affine cone over the Grassmannian $\mathrm{Gr}(2,n+1)$ and consider the torus $H=(\KK^*)^n$ acting on $\overline Y$ by 
\[
 h\cdot e_0\wedge e_i\,=\,h_i e_0\wedge e_i,\qquad h\cdot e_i\wedge e_j\,=\,h_ih_je_i\wedge e_j.
\]
We assert that the normalised limit quotient $\overline Y\nlimquot H$ normalises the following iterated blow-up of $\PP_1^{n-1}$. We set $N_2:=\{2,\ldots,n\}$ and consider a subset $A\subseteq N_2$ with at least two elements. Labeling by $T_2,S_2,\ldots,T_n,S_n$ the homogeneous coordinates of $\PP_1^{n-1}$ we associate to $A$ the subscheme of $\PP_1^n$ given by the ideal
\[
\big\langle
\,T_i^2,\ T_jS_k\,-\,T_kS_j;\ \ i,\,j,\,k\in A,\;j<k\,
\big\rangle.
\]
The collection $\mathcal X$ of corresponding subschemes $X_A$ comes with a partial order given by the schme-theorectic inclusions with $X_{N_2}$ being the minimal element. A {\it linear extension} of this partial order is a total order on $\mathcal X$ which is compatible with the partial order.

\begin{theorem}
\label{thm:blow-up}
Fix a linear extension of the partial order on $\mathcal X$. Then the normalised limit quotient $\overline Y\nlimquot H$ normalises the blow-up of $\PP_1^{n-1}$ in all the subschemes $X_A$ (i.e. their respective proper transforms) in ascending order.
\end{theorem}

Recall that the above action stems from the action of $\mathbb G_a$ on $X=\PP_1^n$ as shown in Sections \ref{sec:nr_limit} and \ref{sec:N0m}. Moreover, keep in mind that the enveloped quotients $V_i$ of $X$ are only subsets of the Mumford quotients of $\overline Y$. Hence the non-reductive limit quotient $X\limquot \mathbb G_a$ in general only is a subset of the reductive limit quotient. This is reflected in the second step of the following procedure to obtain $X\nlimquot \mathbb G_a$.

\begin{theorem}
\label{thm:blow-up_2}
The normalised limit quotient $X\nlimquot \mathbb G_a$ can be obtained by the following procedure.
\begin{enumerate}
 \item Let $X_1$ be the blow-up of $\PP_1^{n-1}$ in the subscheme $X_{N_2}$.
 \item Let $X'_1:=X_1\setminus E$ be the quasiprojective subvariety of $X_1$ where $E$ is the intersection of the proper transform of $V(T_2,\ldots, T_n)\subseteq \PP_1^{n-1}$ with the exeptional divisor in $X_1$. 
 \item Fix a linear extension of the partial order on $\mathcal X$ and blow up $X_1'$ in the respective proper transforms of the remaining subschemes $X_A$, $A\subsetneq N_2$ in ascending order.
 \item Normalise the resulting space.
\end{enumerate}
\end{theorem}

We briefly outline the structure of our proof. For this consider $E_n$ the identity matrix and
\[
 Q:=(E_n,D_n),
 \qquad\text{where}\qquad
 D_n:=(e_j\,+\,e_k)_{1\le j<k\le n}.
\]
Note that $Q$ is the matrix recording the weights of the coordinates of the above $H$-action. We denote the first $n$ columns of $Q$ by $w_{0i}$ and the remaining ones by $w_{jk}$. Furthermore, we fix a Gale dual matrix $P$ of $Q$, i.e. a matrix with $PQ^t=0$, and analogously write $v_{0i},\,v_{jk}$ for its columns. Denoting by $T$ the dense algebraic torus of $\bigwedge^2\KK^{n+1}$ we recall from Section~\ref{sec:nr_limit} that there is a normalisation map
\[
 \overline Y\nlimquot H\ \to\ \overline{\left((\overline Y\cap T)/H\right)^\Sigma},
\]
where the latter is the closure in the toric variety associated to the fan $\Sigma:=\mathrm{GKZ}(P)$. With this the proof of Theorem~\ref{thm:blow-up} will be split into two parts. As a first step we will prove that the blow-up of $\PP_1^{n-1}$ in the subscheme $X_{N_2}$ yields one of the Mumford quotients $X_1$ of $\overline Y$. This quotient comes with a canonical embeddeding into a simplicial toric variety $Z_1$, which arises from a simplicial fan $\Sigma_1$ with rays generated by the columns of $P$. Finally we show that the iterated stellar subdivision of $\Sigma_1$ and the fan $\Sigma$ share a sufficiently large subfan. This implies that the proper transform of $X_1$ under the corresponding toric blow-ups and the limit quotient $\overline Y\limquot H$ share a common normalisation.

In the case $n=2$ the normalised limit quotient is the projective line. If we consider three distinct points the resulting normalised limit quotient is the unique non-toric, Gorenstein, log del Pezzo $\KK^*$-surface of Picard number $3$ and a singularity of type $A_1$, see \cite[Theorem 5.27]{elaines_diss}. The standard construction of this surface is the blow-up of three points on $\PP_2$ followed by the contraction of a $(-2)$-curve. However, we realise it as a single (weighted) blow-up of $\PP_1\times \PP_1$ in the subscheme associated to $\langle\,T_2^2,\;T_3^2,\;T_2S_3-T_3S_2\,\rangle$ where $T_2,S_2,T_3,S_3$ are the homogeneous coordinates on $\PP_1\times \PP_1$. Similar to $\overline M_{0,5}$ which is isomorphic to a single Mumford quotient of the cone over the Grassmannian $\mathrm{Gr}(2,5)$, this surface arises as Mumford quotient of the cone over the Grassmannian $\mathrm{Gr}(2,4)$. For higher $n$ an analogous Mumford quotient needs to be blown up as described above to obtain the limit quotient.

\subsection*{Step 1}
Recall that each chamber in the GIT-fan $\Lambda_H(\overline Y)$ gives rise to a set of semistable points admitting a Mumford quotient. We define two particular chambers and look at their respective quotients. For this consider the following linear forms on $\QQ^n$:
\[
 f_1:=e_1^*-\sum_{i\not=1}e_i^*;\qquad f_{1j}:=e_1^*+e_j^*-\sum_{i\not=1,j}e_i^*.
\]
The zero sets of these linear forms are precisely the walls arising from the partitions $\{\{1\},N\setminus\{1\}\}$ and $\{\{1,j\},N\setminus\{1,j\}\}$ of $N=\{1,\ldots, n\}$ in the sense of Section~\ref{sec:N0m}. We define the following two full dimensional cones in the GIT-fan
\begin{align*}
 \lambda_{0}\ &:=\ \Omega\cap \{w\in\QQ^n;~f_1(w)\ge 0\},\\
 \lambda_{1}\ &:=\ \Omega\cap \{w\in\QQ^n;~f_1(w)\le 0,\;f_{1j}(w)\ge 0\text{ for }j=2,\ldots,n\}. 
\end{align*}
where $\Omega$ is the support of $\Lambda_H(\overline Y)$. While $\lambda_1$ lies inside $\Omega^\star=\mathrm{supp}(\Lambda_H(\overline Y^\star))$ the cone $\lambda_0$ does not. The two cones are adjacent in the sense that they share a common facet, namely $\Omega\cap \mathrm{ker}(f_1)$. Now consider the corresponding Mumford quotients $X_i:=\overline Y^{\mathrm{ss}}(\lambda_{i})\gq H$ with $i=0,1$.

\begin{proposition}
\label{pro_ambient_P1}
In the above notation $X_{0}$ is isomorphic to $\PP_1^{n-1}$. Moreover, $X_1$ is isomorphic to the blow-up of $X_{0}$ in the subscheme $X_{N_2}$.
\end{proposition}

Recall that $\lambda_1\in\Lambda_H(\overline Y)$ gives rise to the enveloped quotient $V_1$ which is the image of the restricted Mumford quotient $\overline Y^{\mathrm{ss}}(\lambda_1)\cap \overline Y'\,\to\,X_1$. 

\begin{proposition}
 \label{pro:enveloped_quotient}
 Let $E$ denote the intersection of the exceptional divisor of $X_1\to X_0$ with the proper transform of $V(T_2,\ldots, T_n)$. Then the enveloped quotient $V_1$ is given by $X_1\setminus E$. In particular, it is quasiprojective. 
\end{proposition}

\begin{proposition}
\label{pro_ambient_P1_extras}
Let $A\subseteq N_2$ be a subset with at least two elements. Then the cone $\mathrm{cone}(v_\eta;\ \eta\subseteq A\cup\{0\})$ lies in $\Sigma_0$. Moreover, consider the ray
\[
 \nu\quad:=\quad 
 \mathrm{cone}\,\left(
 \sum_{i\in A}\;v_{0i}+2\sum_{\eta\subseteq A}v_\eta
 \right)
\]
in the relative interior of the above cone. Let $X'$ be the proper transform of $X_0$ under the blow-up corresponding to the stellar subdivision of $\Sigma_0$ in $\nu$. Then $X'$ is isomorphic to the blow-up of $X_0$ in the subscheme $X_A\subseteq X_0$ given by
\[
 \big\langle
 \,T_i^2,\ T_jS_k-T_kS_j\;;\ \ i,\,j,\,k\in A,\ j<k\,
 \big\rangle.
 \]
\end{proposition}


Before we look at our case we recall the connection between blow-ups and stellar subdivisions in general. For this let $\Sigma_1\to\Sigma_0$ be the stellar subdivision of a simplicial fan in $\ZZ^n$ in the ray $\nu$. To any homogeneous ideal in the Cox ring of the corresponding toric variety $Z_0$ we can associate a subscheme of $Z_0$ in the sense of Cox, see \cite[Section~3]{homogeneous_coordinate_ring} for details. We now ask for an ideal and the associated subscheme $Z_\nu$ such that the blow-up of $Z_0$ in $Z_\nu$ is isomorphic to the toric variety $Z_1$ corresponding to $\Sigma_1$.

For this short reminder set $P$ as the matrix mapping the standard basis vectors $f_i$, $i=1,\ldots,r$ of $F:=\ZZ^r$ to the primitive lattice vectors $v_i\in\ZZ^n$ in the rays of $\Sigma_0$. Then the Cox ring of $Z_0$ is $\KK[E\cap \gamma]$ where $E:=F^*$ and $\gamma$ is the positive orthant in $E\otimes_\mathbb Z \QQ$. If the ray $\nu$ lies in the support of $\Sigma_0$, then there exists a subset $I\subseteq\{1,\ldots,r\}$ and minimal positive integers $\alpha_i\in\ZZ_{\ge 1}$ such that
\[\nu\ =\ \mathrm{cone}\left(\sum_{i\in I}\alpha_iv_i\right)\]
holds. Denoting by $(e_1,\ldots, e_r)$ the dual basis of $(f_1,\ldots, f_r)$ we set 
\[
 E_I:=\mathrm{cone}(e_i;\;i\in I),\qquad
 f:=\sum_{i\in I}\alpha_if_i\in F,\qquad
 c:=\mathrm{lcm}(\alpha_i;~i\in I).
 \]
 We obtain a homogeneous ideal in the Cox ring of $Z_0$ and from it a subscheme $Z_\nu$ of $Z_0$ by
\[
 \langle \chi^e;~e\in E_I,\,\langle\, e,f\,\rangle=c\,\rangle
 \ \subseteq\ 
 \KK[E\cap\gamma].
\]

\begin{proposition}
\label{pro:toricblowup}
 Let $\Sigma_1\to\Sigma_0$ be the above stellar subdivision of the simplicial fan $\Sigma_0$ in $\nu$. If $Z_0,Z_1$ are the toric varieties arising from $\Sigma_0,\Sigma_1$ respectively, then $Z_1$ is isomorpic to the blow-up of $Z_0$ in the subscheme $Z_\nu$.
\end{proposition}

 
We turn back to our setting. We prove Propositions~\ref{pro_ambient_P1},~\ref{pro:enveloped_quotient} and \ref{pro_ambient_P1_extras} using the method of ambient modifications, see \cite[Proposition~6.7]{cox_and_combi_2}. For this note that $X_0$ and $X_1$ come with canonical embeddings into simplicial toric varieties. We provide an explicit construction, for the general case see \cite[Chapter~III, Section~2.5]{cox_rings}. For the index sets we use the same notation as in Section~\ref{sec:N0m}:
\[
 \mathbf N=\{\{i,j\},\ 1\le i<j\le n\},\qquad
 \mathbf N_0=\{\{i,j\},\ 0\le i<j\le n\}.
\]
Viewing $\bigwedge^2\KK^{n+1}$ as the toric variety arising from the positive orthant $\delta$ in $\bigwedge^2\QQ^{n+1}$ we define a subset as follows. We set
\[
  \mathrm{envs}(\lambda_i)\ =\ 
\{
I\subseteq\mathbf N_0;\ 
J\subseteq I,\ 
\lambda_i^\circ\subseteq \omega_J^\circ\subseteq \omega_I^\circ
\text{ for some }\overline Y\text{-set }J
\}
\]
as the collection of {\it enveloping sets}. Denoting by $f_\eta$ with $\eta\in\mathbf N_0$ the standard basis vector in $\bigwedge^2\QQ^{n+1}$ we consider the subfan of $\delta$
\[
 \hat \Sigma_i\ :=\ 
 \{
 \mathrm{cone}(f_\eta;\ \eta\in J);\ J\subseteq\mathbf N_0\setminus I\text{ for some }I\in\mathrm{envs}(\lambda_i)
 \}
\]
and the corresponding toric variety $\hat Z_i\subseteq\bigwedge^2\KK^{n+1}$. Then $\hat Z_i$ admits a good quotient $\hat Z_i\to Z_i$; the quotient space is toric again and the quotient morphism corresponds to the lattice homomorphim $P\colon\ZZ^{{n+1}\choose 2}\to\ZZ^{n\choose 2}$. The fan of $Z_i$ is given by
\[
\Sigma_i\ =\ 
\{\mathrm{cone}(v_\eta;\,\eta\in \mathbf N_0\setminus I);\ I\in\mathrm{envs}(\lambda_i)\}
\]
We now turn to the embedded spaces. Starting with the embedding $\overline Y\subseteq\bigwedge^2\KK^{n+1}$ we have $\overline Y\cap\hat Z_i=\overline Y^{\mathrm{ss}}(\lambda_i)$ and the quotient $\hat Z_i\to Z_i$ restricts to the good quotient ${\overline Y^{\mathrm{ss}}(\lambda_i)}\to X_i$. The situation fits into the following commutative diagram where the vertical arrows are closed embeddings.
\[
 \xymatrix{
 {\bigwedge\nolimits^2\KK^{n+1}}
 &
 {\hat Z_i}
 \ar[r]
 \ar[l]
 &
 Z_i
 \\
 {\overline Y}
 \ar[u]
 &
 {\overline Y^{\mathrm{ss}}(\lambda_i)}
 \ar[u]
 \ar[l]
 \ar[r]
 &
 X_i
 \ar[u]
 }
\]

\begin{proof}[Proofs of Proposition \ref{pro_ambient_P1}, \ref{pro:enveloped_quotient} and \ref{pro_ambient_P1_extras}]
We prove the first part of Proposition~\ref{pro_ambient_P1_extras}. For this we set $J:= \mathbf N_0\setminus \{\eta;\ \eta\subseteq A\cup\{0\}\}$. With Proposition~\ref{pro:ffaces} it is easy to see, that $J$ is a $\overline Y$-set. Moreover, $\lambda_0^\circ\subseteq\omega_J^\circ$ holds. By definition of $\Sigma_0$ it is now clear that it contains $\mathrm{cone}(v_\eta;\ \eta\subseteq A\cup\{0\})$.

We now perform the ambient modification. For this note that the weight $w_{01}$ is extremal in $\Lambda_H(\overline Y)$, hence we can contract $v_{01}$. It can be written as a non-negative linear combination
\[
 v_{01}=\sum_{\eta\in \mathbf N_0}\alpha_\eta v_{\eta},
 \quad\text{where}\quad 
 \alpha_\eta=
 \mbox{
 {\tiny 
 $\begin{cases}
  0&\text{if } 1\in\eta\\
  1&\text{if } 0\in\eta,1\notin\eta\\
  2&\text{else}
 \end{cases}$
 }.
 }
\]
In particular, it lies in the above cone $\mathrm{cone}(v_\eta;~\eta\subseteq \{0\}\cup N_2\})$. The total coordinate spaces of the embedding toric varieties $Z_0$ and $Z_1$ are affine spaces, they are given by
\[
 \overline Z_0 =\KK^{\mathbf N_0\setminus\{0,1\}}
 \quad\text{and}\quad
 \overline Z_1 = \bigwedge\nolimits^2\KK^{n+1}=\KK^{\mathbf N_0}.
\]
Furthermore, the ambient modification $\Sigma_1\to\Sigma_0$ gives rise to a morphism of the total coordinate spaces of the respective toric varieties
\[
 c\colon\overline Z_1\to\overline Z_0;\qquad (x_\eta)_{\eta\in \mathbf N_0}\ \mapsto\ (x_{01}^{\alpha_\eta}\;x_\eta)_{\eta\in\mathbf N_0\setminus\{0,1\}}.
\]
We label the variables of the total coordinate space $\overline Z_{0}$ by $S_\eta$ where $\eta$ runs through $\mathbf N_0\setminus \{0,1\}$. Recall that we have a closed embedding $\overline Y\subseteq\overline Z_1$. The vanishing ideal of the image $\overline X_{0}:=c(\overline Y)$ in the Cox ring is given as
\[
 \langle\, S_{ij}\,-\,S_{0i}S_{1j}\,+\,S_{0j}S_{1i}\;;\ \ 2\le i<j\le n\,\rangle
 \quad\subseteq\quad
 \mathcal R(Z_{0}).
\]
It turns out that $\overline X_0$ is in fact isomorphic to the affine space via
\[
\iota\colon \KK^{n-1}\times\KK^{n-1}\to\overline Z_0\;\qquad
(x,y)\mapsto (x,y,(x_iy_j-x_jy_i)_{i< j}).
\]
The original $H$-action on $\overline Y$ descends via $\iota^{-1}\circ c$ to $\KK^{n-1}\times\KK^{n-1}$ and is explicitly given by the weight matrix $Q_0=[E_{n-1},E_{n-1}]$ where $E_{n-1}$ is the identity matrix. This shows that $X_0$ is isomorphic to $\PP_1^{n-1}$. For convenience we summarise the situation in the following commutative diagram.
\[
 \xymatrix{
 {\overline Z_1}
 \ar[r]^c
 &
 {\overline Z_0}
 \\
 {\overline Y}
 \ar[r]^c
 \ar[u]
 &
 {\overline X_0}
 \ar[u]
 &
 {\KK^{2(n-1)}}
 \ar[ul]_\iota
 \ar[l]_\iota
 }
\]

The next step of the proof is the second half of Proposition~\ref{pro_ambient_P1_extras}. From Proposition~\ref{pro:toricblowup} we infer that the ideal in $\mathcal O(\overline Z_0)=\mathcal R(Z_0)$ yielding the center of the blow-up is given by
\[
\langle\ S_{0i}^2,\ S_\eta;\ \ i\in A,\ \eta\subseteq A\ \rangle.
\]
If we pullback this ideal via $\iota^*$ (see \cite[Lemma~2.1]{diagonal}), then in homogeneous coordinates over $\PP_1^{n-1}$ we obtain 
\[
 \langle\ T_i^2,\ T_jS_k-T_kS_j;\ \ i,j,k\in A,\ j<k\rangle.
\]
In the case of the ambient modification of Proposition~\ref{pro_ambient_P1} we set $A=N_2$ to obtain the assertion. Finally, we turn to Proposition~\ref{pro:enveloped_quotient} and determine the enveloped quotient. For this recall that the image of the categorical quotient in Section~\ref{sec:N0m} was given by $\overline Y'=(\overline Y\setminus \overline Y^\star)\cup\{0\}$, see Proposition~\ref{pro:Im_of_kappa}. This means that the enveloped quotient $V_1\subseteq X_1$ is given as the image of 
\[
 \pi\colon\overline Y^{\mathrm{ss}}(\lambda_1)\setminus D\quad\to\quad X_1,
\]
where $D:=V(S_{0i};~i=1,\ldots,n)\subseteq\bigwedge^2\KK^{n+1}$. The quotient is geometric, hence the enveloped quotient is $V_1=X_1\setminus \pi(D)$. Now consider the subvariety $V(T_2,\ldots, T_n)\subseteq\PP_1^{n-1}$. Transferring it via $\iota$ and then taking the proper transform we obtain the subvariety of $X_1$ given by $\langle\,S_{02},\ldots,S_{0n}\,\rangle$ in the Cox ring $\mathcal R(Z_1)$. The intersection with the exceptional divisor is precisely the set $E=\pi(D)$.

\end{proof}



\subsection*{Step 2}
In this step we show that the remaining blow-ups lead to the limit quotient $\overline Y\limquot H$. As before, $Q=(E_n,D_n)$ is the matrix recording the weights of the coordinates of the $H$-action and we label its columns by $w_\eta$ with $\eta\in\mathbf N_0$ and $\mathbf N_0=\{\{i,j\};~0\le i<j\le n\}$. We then have the Gale dual matrix $P$ with columns denoted by $v_\eta$. Moreover, $\Sigma_1$ is the simplicial fan in $\ZZ^{n\choose 2}$ from the preceeding step and we recall that $X_1=\overline Y^{\mathrm{ss}}(\lambda_1)\gq H$ is embedded into the corresponding toric variety $Z_1$. 

Now let $R=\{A_1,A_2\}$ be a {\it true} two-block partition of $N$, i.e. a partition with $|A_1|,\,|A_2|\ge 2$. To every such partition we associate a ray
\[
 \nu_R
 \quad := \quad
 \mathrm{cone}\left(\sum_{i\in A_1}\!v_{0i}\;+\;2\!\!\!\sum_{j<k\in A_1}\!\!\!v_{jk}\right)
 \quad=\quad
 \mathrm{cone}\left(\sum_{i\in A_2}\!v_{0i}\;+\;2\!\!\!\sum_{j<k\in A_2}\!\!\!v_{jk}\right).
\]
Clearly, there exists $A_R\in\{A_1,A_2\}$ with $1\not\in A_R$. From Proposition~\ref{pro_ambient_P1_extras} we now infer that the cone $\sigma_{R}:=\mathrm{cone}(v_{\eta};~\eta\subseteq\{0\}\cup A_R)$ containing $\nu_R$ in its relative interior lies in $\Sigma_1$.

Note that no two rays lie in the relative interior of the same cone of $\Sigma_1$. The above defined collection of rays hence comes with a natural partial order inherited from the fan $\Sigma_1$:
\[
 \nu_{R}\le\nu_{S}:\iff\sigma_R\seite\sigma_S\iff {A}_{R}\subseteq {A}_{S}.
\]
We choose a linear extension of this partial order. Beginning with the maximal ray we then consider the iterated stellar subdivision of $\Sigma_1$ in all the rays in descending order. The resulting fan we denote by $\Sigma_r$.

While it is not true that $\Sigma_r$ coincides with the GKZ-decomposition $\Sigma=\mathrm{GKZ}(P)$, both fans share a sufficiently large subfan. To make this precise let $T$ be the dense torus of $\bigwedge^2\KK^{n+1}$. To $\overline Y\cap T$ we can associate its {\it tropical variety} $\mathrm{Trop}(\overline Y\cap T)$, which is the support of a quasifan in $\bigwedge^2\QQ^{n+1}$. For a detailed description of this space see \cite{tropical_grassmannian}. For our purposes it suffices to know that the image $\Delta:=P(\mathrm{Trop}(\overline Y\cap T))$ intersects the relative interior $\mathrm{cone}(v_\eta;~\eta\in J)^\circ$ of a cone if and only if $\mathbf N_0\setminus J$ is a $\overline Y$-set, see \cite[Proposition~2.3]{subvarieties_of_tori}. We now define the {\it $\Delta$-reduction} of $\Sigma$ as the fan
\[
 \Sigma^\Delta:=\{\sigma;~\sigma\seite\tau\in\Sigma\text{ for some }\tau \text{ with }\tau^\circ\cap \Delta\not=\emptyset\}.
\]
Note that the relative interiors of all maximal cones of $\Sigma^\Delta$ intersect $\Delta$. Moreover, by \cite[Proposition~2.3]{subvarieties_of_tori} the closure of $(\overline Y\cap T)/H$ in the toric variety corresponding to $\Sigma$ is already contained in the toric subvariety defined by $\Sigma^\Delta\subseteq\Sigma$.

\begin{proposition}
\label{pro_gkz_is_nested}
 The $\Delta$-reduction $\Sigma^\Delta$ is a subfan of $\Sigma_r$. 
\end{proposition}

\begin{corollary}
\label{cor:gkz_is_nested}
The proper transform of the Mumford quotient $X_1\subseteq Z_1$ under the toric morphism arising from $\Sigma_r\to\Sigma_1$ and the limit quotient $\overline Y\limquot  H$ share a common normalisation.
\end{corollary}
\begin{proof}
 For this just note that the following closures coincide and the first morphism is the normalisation map.
 \[
 \overline Y\nlimquot H
  \ \to\ 
  \overline{\left( (\overline Y\cap T)\,/\,H\right)^{\Sigma}}
  \ =\ 
  \overline{\left( (\overline Y\cap T)\,/\,H\right)^{\Sigma^\Delta}}
  \ =\ 
  \overline{\left( (\overline Y\cap T)\,/\,H\right)^{\Sigma_r}}.
 \]
\end{proof}
\begin{remark}
 In fact, with only minor modifications the Step 2 works for every Mumford quotient of $\overline Y$ which arises from a fulldimensional chamber $\lambda$ lying in $\Omega^\star$.
\end{remark}

The idea of the proof of Proposition~\ref{pro_gkz_is_nested} is to give a combinatorial description of the cones in $\Sigma^\Delta$ and to show that these are geometrically nested in the sense of Section~\ref{sec:combinatorics}.

For the moment let $Q\in \Mat(k,r;\ZZ)$ and $P\in\Mat(n,r;\ZZ)$ be arbitrary Gale dual matrices. We set $R:=\{1,\ldots, r\}$. For a subset $I\subseteq R$ we denote by $\gamma_I\subseteq\QQ^n$ the cone generated by the $e_i,~i\in I$ and by $\omega_I:=Q(\gamma_I)$ its image under $Q$. Moreover, if $v_i,\ i\in R$ are the columns of $P$ we set $\sigma_J:=\mathrm{cone}(v_j;\ j\in J)$. A system $\mathfrak B$ of subsets of $R$ is a {\it separated} $R${\it -collection} if any two $I_1,I_2\in\mathfrak B$ admit an invariant separating linear form $f$, in the sense that
\[
 P^*(\QQ^n)\subseteq \mathrm{ker}(f),\quad
 f_{|\gamma_{I_1}}\ge 0,\quad
 f_{|\gamma_{I_2}}\le 0,\quad
 \mathrm{ker}(f)\cap \gamma_{I_i}=\gamma_{I_1}\cap\gamma_{I_2}.
\]
The separated $R$-collections come with a partial order; for two $R$-collections $\mathfrak B_1,\mathfrak B_2$ we write $\mathfrak B_1\le \mathfrak B_2$ if for every $I_1\in \mathfrak B_1$ there exists $I_2\in \mathfrak B_2$ such that $I_1\subseteq I_2$ holds. A separated $R$-collection $\mathfrak B$ will be called {\it normal} if it cannot be enlarged as an $R$-collection and the cones $\omega_I,~I\in \mathfrak B$ form the normal fan of a polyhedron. With respect to the above partial order there exists a unique maximal normal $R$-collection, namely $\langle R\rangle$ which consists of all subsets which are invariantly separable from $R$. By $\mathcal M$ we denote the submaximal normal $R$-collections in the sense, that $\langle R\rangle$ is the only dominating normal $R$-collection. Finally, for a fixed normal $R$-collection $\mathfrak B$ let $\mathcal M(\mathfrak B)$ consist of those collections of $\mathcal M$ lying above $\mathfrak B$.

If $P$ consists of pairwise linearly independent columns, then by \cite[Section II.2]{cox_rings} there is an order reversing bijection
\[
 \{\text{normal }R\text{-collections}\}\ \to\ \Sigma;
 \qquad
 \mathfrak B\ \mapsto\ \bigcap_{I\in\mathfrak B}\sigma_{R\setminus I}.
\]
where again $\Sigma=\mathrm{GKZ}(P)$ is the GKZ-decomposition. It is clear that each maximal $R$-collection $\mathfrak A\in\mathcal M$ gives rise to a ray $\nu_{\mathfrak A}=\bigcap_\mathfrak A\sigma_{R\setminus I}$ of $\Sigma$.
\begin{proposition}
\label{pro:rays_general}
Let $\mathfrak B$ be a normal $R$-collection. Then the cone corresponding to $\mathfrak B$ can be written as
\[
 \bigcap_{I\in\mathfrak B}\sigma_{R\setminus I}
 \ = \ 
 \mathrm{cone}\left(\nu_\mathfrak A;~\mathfrak A\in\mathcal M(\mathfrak B) \right).
\]
\end{proposition}

\begin{proof}
 From the order reversing property of the above bijection it is clear, that every ray $\nu_\mathfrak A$ with $\mathfrak A\in\mathcal M(\mathfrak B)$ lies in $\sigma:=\bigcap_\mathfrak B \sigma_{R\setminus I}$. Moreover, there must exists a set of maximal $\gamma$-collections $\mathcal N\subseteq\mathcal M$ such that the extremal rays of $\sigma$ are precisely the $\nu_\mathfrak A$ with $\mathfrak A\in\mathcal N$. Again from the above bijection we know that this means $\mathfrak A\ge \mathfrak B$. The assertion then follows from the maximality of $\mathfrak A$.
\end{proof}

We now return to our special case where $Q=(E_n,D_n)$ holds and the index set $R$ equals $\mathbf N_0$. We are interested in a description of the submaximal collections $\mathcal M(\mathfrak B)$ where $\mathfrak B$ consists of $\overline Y$-sets. The reason is the following Proposition.

\begin{proposition}
\label{pro:cones_of_sigma_s1}
 Let $\mathfrak B$ be a normal $\mathbf N_0$-collection and suppose that its associated cone $\bigcap_{I\in\mathfrak B}\sigma_{\mathbf N_0\setminus I}$ is a maximal cone in $\Sigma^\Delta$. Then $\mathfrak B$ is a collection of $\overline Y$-sets.
\end{proposition}

\begin{proof}
Since $(\bigcap_{\mathfrak B}\sigma_{\mathbf N_0\setminus I})^\circ\cap \Delta\not=\emptyset$ holds the same is true for every $\sigma_{\mathbf N_0\setminus I}^\circ$ with $I\in\mathfrak B$. By \cite[Proposition~2.3]{subvarieties_of_tori} this implies that $\mathfrak B$ is a collection of $\mathfrak F$-faces.
\end{proof}

\begin{proposition}
\label{pro:maximal_collections_over_ffaces}
 Suppose that $\mathfrak B$ is a normal $\mathbf N_0$-collection of $\overline Y$-sets and $\mathfrak A\in \mathcal M(\mathfrak B)$ is a submaximal collection dominating it. Then $\mathfrak A$ is of either one of the following types.
 \begin{enumerate}
  \item The collections $\langle I\rangle$ where $I:=\mathbf N_0\setminus \{\eta\}$ for some $\eta\in\mathbf N$.
  \item The collections $\langle I_1,I_2\rangle$ where $I_i:=\{\eta;\ \eta\cap A_i\not=\emptyset\}$ for a two-block partition $R=\{A_1,A_2\}$ of $N$.
 \end{enumerate}
 Moreover, if a collection of the second type lies over $\mathfrak B$, then $\mathfrak B$ contains the set $J_0:=\{\eta;\ A_i\cap \eta\not=\emptyset\text{ for }i=1,2\}$.
\end{proposition}

Since every collection of the first type is uniquely determined by the element $\eta$, we write it as $\mathfrak A_\eta$. The ray $\rho_\eta$ of $\Sigma$ corresponding to this submaximal collection is generated by $v_\eta$.
 
If a submaximal collection is of the second type, then it is characterised by the partition $R$ of $N$; for it we write $\mathfrak A_R$. Moreover, the associated ray arises as intersection of $\sigma_{\mathbf N_0\setminus I_1}$ and $\sigma_{\mathbf N_0\setminus I_2}$. We now have to discriminate two cases. If the partition $R$ is of the form $[i]:=\{\{i\},N\setminus\{i\}\}$, then the corresponding ray $\rho_{[i]}=\rho_{0i}$ is generated by $v_{0i}$. Otherwise, if $R$ is a true two-block partition, by \cite[Proposition~4.1]{on_chow_quotients} we know that this ray is precisely $\nu_R$, which was defined at the beginning of Step 2.

\begin{proof}[Proof of Proposition~\ref{pro:maximal_collections_over_ffaces}]
Consider an $I\in \mathfrak A$ such that $\omega_I$ is full dimensional. We now discrimnate two cases. For the first case assume that $\omega_I=\Omega$ holds. Since $\mathfrak A$ is submaximal, $I=\mathbf N_0\setminus\{\eta\}$ for some $\eta\in\mathbf N_0$. If we had $0\in \eta$, then $\omega_I$ would be a proper subset of $\Omega$.

We turn to the second case where $\omega_I\subsetneq \Omega$ holds. Then there exists an $I'\in\mathfrak A$ such that $\omega_{I'}$ is a facet of $\omega_I$ and $\omega_{I'}^\circ\cap\Omega^\circ$ is non-empty. Since $\mathfrak B$ cannot be enlarged as $\mathbf N_0$-collection, there moreover exist $J,J'\in\mathfrak B$ such that 
\[
 \omega_J^\circ\subseteq\omega_I^\circ,
 \qquad
 \omega_{J'} \text{ is a facet of }\omega_J,
 \qquad
 \omega_{I'}^\circ\cap\omega_{J'}^\circ\not=\emptyset.
\]
 Now $\omega_{J'}$ is a subset of one of the walls of $\Lambda_H(\overline Y)$. Thus, from Theorem~\ref{thm:fan} we know that there exists some partition $\{A_1,A_2\}$ of $N$ such that $J'$ is a subset of $J_0:=\{\eta;\ A_i\cap\eta\not=\emptyset\text{ for }i=1,2\}$. We now claim that $J'$ equals $J_0$.

For this let $i_1\in A_1,i_2\in A_2$ be two indices. Since $\omega_{J'}$ is of dimension $n-1$, there exist $i_1'\in A_1,\ i_2'\in A_2$ such that $\{i_1,i_2'\}$ and $\{i_2,i_1'\}$ lie in $J'$. From the inclusion $J'\subseteq J_0$ we know that $\{i_1,i_1'\}$ does not lie in $J'$, hence from the characterisation of $\overline Y$-sets in Proposition~\ref{pro:ffaces} it follows that $\{i_1,i_2\}$ lies in $J'$. This proves our claim.

Now let $\mathfrak A'$ be the normal $R$-collection consisting of all faces which are invariantly separable from
\[
 \{\eta;~\eta\cap A_1\not=\emptyset\}
 \qquad\text{and}\qquad
 \{\eta;~\eta\cap A_2\not=\emptyset\}.
\]
Then $\mathfrak A'$ is submaximal and the assertion follows if we show that $\mathfrak A\le\mathfrak A'$ holds. For this note that $\omega_{J'}$ is the intersection of $\Omega$ with the zero set of 
\[
 l:=\sum_{i\in A_1}e^*_i-\sum_{i\in A_2}e^*_i.
\]
Since the collection $\{\omega_K;\ K\in\mathfrak A\}$ forms a fan with support $\Omega$, for every cone $\omega_K,\ K\in \mathfrak A$ we have $l_{|\omega_K}\ge 0$ or $l_{|\omega_K}\le 0$. This implies that $\mathfrak A\le \mathfrak A'$ holds.

\end{proof}

Recall that we want show that the (maximal) cones of $\Sigma^\Delta$ are geometrically nested in the sense of Section~\ref{sec:combinatorics} and hence lie in $\Sigma_r$. The relevant property of the corresponding $\mathbf N_0$-collections shall be discusses in the sequel.

Let $R=\{A_1,A_2\}$ and $S=\{B_1,B_2\}$ be two-block partitions of $N$ and $\eta\in \mathbf N_0$. We then call the pair $\{\eta, R\}$ {\it compatible} if $\eta$ lies in $A_1$ or in $A_2$. Moreover, we call $\{R,S\}$ {\it compatible}, if there exist $i,j\in\{1,2\}$ such that $A_i\subseteq B_j$ holds. The pairs of submaximal collections $\{\mathfrak A_\eta,\mathfrak A_R\}$ and $\{\mathfrak A_R,\mathfrak A_S\}$ are {\it compatible}, if the corresponding pairs $\{\eta,R\}$ and $\{R,S\}$ are compatible.

\begin{proposition}
\label{pro:cones_of_sigma_s2}
 Let $\mathfrak B$ be a normal $\mathbf N_0$-collection of $\overline Y$-sets. Then the submaximal collections in $\mathcal M(\mathfrak B)$ are pairwise compatible.

\end{proposition}

\begin{proof}
Let $\mathfrak A_\eta, \mathfrak A_R\ge \mathfrak B$ be two submaximal collections with $R=\{A_1,A_2\}$. Then $\{i,j\}:=\eta$ is contained in no $I\in\mathfrak B$. However, the cones $\omega_I,\ I\in \mathfrak B$ cover $\Omega$.  Since $w_{ij}=w_{0i}+w_{0j}$ is the only positive linear combination of $w_{ij}$, the sets $\{0,i\},\; \{0,j\}$ must lie in a common $I\in\mathfrak B$. From the characterisation in Proposition~\ref{pro:maximal_collections_over_ffaces}(ii) we can now infer that without loss of generality $i,j\in A_1$ holds and this implies compatibility of $\eta$ with $R$.

Suppose we have $\mathfrak A_R,\mathfrak A_S\ge \mathfrak B$ with $R=\{A_1,A_2\}$ and $S=\{B_1,B_2\}$. From Proposition~\ref{pro:maximal_collections_over_ffaces} we infer that the set $J_0=\{\eta;\ \eta\cap A_i\not=\emptyset\text{ for }i=1,2\}$ lies in $\mathfrak B$. This means that $J_0$ lies in one of the maximal sets of $\mathfrak A_S$. In other words, there exists $j$ such that
\[
 \eta\cap A_1\not=\emptyset\quad\text{and}\quad\eta\cap A_2\not=\emptyset
 \qquad\Longrightarrow \qquad
 \eta\cap B_j\not=\emptyset.
\]
This implies that there exists $i$ such that $A_i\subseteq B_j$ holds and hence $\{R,S\}$ is compatible.
\end{proof}

The final thing we show is that the cones defined by compatible submaximal collections are geometrically nested in the sense of Section~\ref{sec:combinatorics}. For this we define $\mathbf S$ as the collection of two-block partitions of $N$ and set $\mathbf S_{\ge 2}$ as the subcollection of true two-block partitions, i.e. the partitions $\{A_1,A_2\}$ with $|A_1|,\,|A_2|\ge 2$.

We set $\mathcal V=\{\rho_\eta;~\eta\in\mathbf N_0\}$ as the set of rays of $\Sigma_1$. Keep in mind that the rays $\rho_{0i}$ stem from the partitions $[i]=\{\{i\},N\setminus\{i\}\}$, hence we have $\rho_{0i}=\rho_{[i]}$.

Moreover, we define $\mathcal S:=\{\sigma_R;\ R\in \mathbf S_{\ge 2}\}$ as the collection of cones in $\Sigma_1$ associated to true two-block partitions. This is precisely the collection of cones containing the rays $\nu_R$ in their relative interiors.

\begin{lemma}
 \label{lem:nested}
 Consider the collection of cones
 \[
  \mathcal C:=\{\rho_\eta,\,\sigma_R;~\mathfrak A_\eta,\mathfrak A_R\in\mathcal N\}
  \qquad\text{for some}\qquad
  \mathcal N\subseteq \{\mathfrak A_\eta,\,\mathfrak A_R;~\eta\in \mathbf N,\, R\in\mathbf S\}.
 \]
If any pair in $\mathcal N$ is compatible, then $\mathcal C$ is geometrically nested in $\mathcal V\cup\mathcal S$.
\end{lemma}

\begin{proof}
Consider a subset $\mathcal H\subseteq\mathcal C$ of imcomparable elements with $|\mathcal H|\ge 2$. Moreover, take $\mathcal S'\subseteq\mathcal S$ to be a non-empty conjunct subset. Assuming that
 \[
  \sigma:=\sum_{\tau\in S'}\tau\ =\ \sum_{\tau\in\mathcal H}\tau\in\Sigma_1
 \]
holds we have to show that there exist an incompatible pair in $\mathcal N$.

Recall that the cones of $\mathcal S=\{\sigma_R;\ R\in \mathbf S_{\ge 2}\}$ have the form
\[
 \sigma_R=\mathrm{cone}(v_\eta;~\eta\subseteq \{0\}\cup A_R)
 \qquad\text{where}\qquad
 1\notin A_R\in R.
\]
%
%
%
Consider two cones $\sigma_1,\sigma_2\in\Sigma_1$ such that their sum lies in $\Sigma_1$ as well. Since $\Sigma_1$ is simplicial, the rays of $\sigma_1+\sigma_2$ are precisely given by the union of the rays of $\sigma_1$ and $\sigma_2$. In particular, if $\rho$ is a ray of some $\tau\in\mathcal H$, then there exists $\tau'\in\mathcal S'$ such that $\rho$ is a ray of $\tau'$. Clearly, the same is true with $\mathcal H$ and $\mathcal S'$ exchanged. 

If $|\mathcal H\cap\mathcal S|=0$ holds, i.e. $\mathcal H$ is a subset of $\mathcal V$, then one easily sees that there exist $\rho_{[i]},\,\rho_{ij}\in\mathcal H$. Clearly, $[i]$ and $\{i,j\}$ are incompatible, hence $\mathfrak A_{[i]},\,\mathfrak A_{\{i,j\}}\in\mathcal N$ are the incompatible partitions.

We consider the case $|\mathcal H\cap \mathcal S|=1$ and denote the single cone in $\mathcal H\cap \mathcal S$ by $\sigma_R$. Since $|\mathcal H|\ge 2$ holds there exists an element $\rho\in\mathcal H\cap \mathcal V$. We distinguish two subcases.

In the first case let this ray be of the form $\rho=\rho_{[i]}$. Then we find $\sigma_S\in\mathcal S'$ with $\rho\seite\sigma_S$. From the special form of the cone $\sigma_R$ we know that there also exists $j\in N$ with $\rho_{ij}\seite\sigma_S$. By the assumption made on $\mathcal H$ we have $\rho_{[i]}\not\seite\sigma_R$; and the special form of $\sigma_R$ then means that also $\rho_{ij}\not\seite\sigma_R$ holds. Hence $\rho_{ij}$ lies in $\mathcal H$ and $\mathfrak A_{[i]},\,\mathfrak A_{\{i,j\}}\in\mathcal N$ are the incompatible collections.

In the second case where $\rho=\rho_{ij}$ holds we again find $\sigma_S\in\mathcal S'$ with $\rho_{ij}\seite\sigma_S$. From the special form of the cone $\sigma_S$ we know that both $\rho_{[i]}$ and $\rho_{[j]}$ are rays of $\sigma_S$. Since $\rho_{ij}\not\seite\sigma_R$ holds, at least one of the rays $\rho_{[i]}$, $\rho_{[j]}$ is not a ray of $\sigma_R$. Without loss of generality this implies that again $\rho_{[i]}$ lies in $\mathcal H$ and $\mathfrak A_{[i]},\,\mathfrak A_{\{i,j\}}\in\mathcal N$ are the incompatible collections.

Now we assume that $|\mathcal H\cap\mathcal S|\ge 2$ holds. Then there exist $\sigma_R,\,\sigma_S\in\mathcal H\cap \mathcal S$. For $\eta,\zeta\in\mathbf N$ let $\rho_\eta\seite\sigma_R$ and $\rho_\zeta\seite\sigma_S$ be rays such that $\rho_\eta\not\seite\sigma_S$ and $\rho_\zeta\not\seite\sigma_R$ hold. Since $\mathcal S'$ is conjunct, we find $\xi_1,\ldots,\xi_r\in\mathbf N$ with 
\[
 \xi_1=\eta,\qquad
 \xi_r=\zeta,\qquad
 \rho_{\xi_i}\seite\sigma
 \qquad\text{and}\qquad
 \xi_i\cap \xi_{i+1}\not=\emptyset.
\]
Let $i'$ be the smallest index, for which $\rho_{\xi_{i'}}$ is not a ray of $\sigma_{R}$. If $\rho_{\xi_{i'}}$ lies in $\mathcal H$, then we know $\rho_{\xi_{i'}}\not\seite\sigma_R$ holds. This means that $\xi_{i'}$ and $R$ are incompatible. If $\rho_{\xi_{i'}}$ does not lie in $\mathcal H$, then there exists an $\sigma_{R'}\in\mathcal H$ such that $\rho_{\xi_{i'}}$ is a ray of $\sigma_{R'}$. This implies that $R$ and $R'$ are incompatible.
\end{proof}

\begin{proof}[Proof of Proposition \ref{pro_gkz_is_nested}]
Consider the cone $\sigma\in\Sigma^\Delta$. In order to show show that $\sigma$ lies in $\Sigma_r$ we can without loss of generality assume that $\sigma$ is maximal. Let $\mathfrak B$ be the associated normal $\mathbf N_0$-collection with
\[
 \sigma\ =\ \bigcap_{I\in\mathfrak B}\sigma_{\mathbf N_0\setminus I}.
\]
By Propositions~\ref{pro:cones_of_sigma_s1}, \ref{pro:cones_of_sigma_s2} we know that $\mathcal M(\mathfrak B)$ is a set of compatible normal $\mathbf N_0$-collections. Furthermore, by Proposition~\ref{pro:rays_general}
\[
 \sigma
 \ =\ 
 \mathrm{cone}(\nu_\mathfrak A;~\mathfrak A\in\mathcal M(\mathfrak B))
 \ =\ 
 \mathrm{cone}(\rho,\,\nu_R;\ \rho\in\mathcal V\cap\mathcal C,\ \sigma_R\in\mathcal S\cap\mathcal C)
\]
holds. Lemma~\ref{lem:nested} shows that $\mathcal C=\{\rho_\eta,\,\sigma_R;~\mathfrak A_\eta,\mathfrak A_R\in\mathcal M(\mathfrak B)\}$ is geometrically nested in $\mathcal V\cup\mathcal S$. And finally, from Proposition~\ref{pro_main_fan} we infer that $\sigma$ lies in $\Sigma_r$.
\end{proof}

\begin{proof}[Proof of Theorem~\ref{thm:blow-up}]
The Theorem now follows directly from Proposition~\ref{pro_ambient_P1} and Corollary~\ref{cor:gkz_is_nested}.
\end{proof}

\begin{proof}[Proof of Theorem~\ref{thm:blow-up_2}]
As in the reductive case we performed the first blow-up in Proposition~\ref{pro_ambient_P1}. In Proposition~\ref{pro:enveloped_quotient} determined the subset of $X_1$ that has to be removed due to the fact that the morphism $\kappa\colon \KK^{2n}\to\bigwedge^2\KK^{n+1}$ is not surjective. Finally the remaining blow-ups are performed as in the reductive case, see Corollary~\ref{cor:gkz_is_nested}.
\end{proof}


\begin{thebibliography}{10}

\bibitem{cox_rings}
I.~Arzhantsev, U.~Derenthal, J.~Hausen, and A.~Laface.
\newblock {\em {Cox rings}}.
\newblock 2011.
\newblock \url{arXiv:1003.4229v2 [math.AG]}.

\bibitem{non_reductive_gelfand}
I.~V. Arzhantsev, D.~Celik, and J.~Hausen.
\newblock Factorial algebraic group actions and categorical quotients.
\newblock {\em J. Algebra}, 387:87--98, 2013.

\bibitem{gitviacox}
I.~V. Arzhantsev and J.~Hausen.
\newblock Geometric invariant theory via cox rings.
\newblock {\em Journal of Pure and Applied Algebra}, 213:154--172, 2009.

\bibitem{diagonal}
H.~B{\"a}ker.
\newblock {On the {C}ox ring of blowing up the diagonal}.
\newblock {preprint: }\url{arXiv:1402.5509}.

\bibitem{on_chow_quotients}
H.~B{\"a}ker, J.~Hausen, and S.~Keicher.
\newblock {On Chow quotients of torus actions}.
\newblock {preprint: }\url{arXiv:1203.3759}.

\bibitem{hausen:amplecone}
F.~Berchtold and J.~Hausen.
\newblock {GIT-Equivalence beyond the ample cone}.
\newblock {\em Michigan Mathematical Journal}, 54(3):483--515, 2006.

\bibitem{homogeneous_coordinate_ring}
D.~Cox.
\newblock The homogeneous coordinate ring of a toric variety.
\newblock {\em J. Algebraic Geom}, 4(1):17--50, 1995.

\bibitem{dolgachevhu:GIT}
I.~V. Dolgachev and Y.~Hu.
\newblock Variation of geometric invariant theory quotients.
\newblock {\em Publications Mathematiques De L Ihes}, 87:5--51, 1998.

\bibitem{towards_non_reductive_git}
B.~Doran and F.~Kirwan.
\newblock Towards non-reductive geometric invariant theory.
\newblock {\em Pure Appl. Math. Q.}, 3(1, part 3):61--105, 2007.

\bibitem{incidence_combinatorics}
E.-M. Feichtner and D.~N. Kozlov.
\newblock Incidence combinatorics of resolutions.
\newblock {\em Selecta Math. (N.S.)}, 10(1):37--60, 2004.

\bibitem{discriminants}
I.~M. Gelfand, M.~M. Kapranov, and A.~V. Zelevinsky.
\newblock {\em Discriminants, resultants and multidimensional determinants}.
\newblock Modern Birkh\"auser Classics. Birkh\"auser Boston Inc., Boston, MA,
  2008.
\newblock Reprint of the 1994 edition.

\bibitem{git_based_on_weil_divisors}
J.~Hausen.
\newblock {Geometric invariant theory based on Weil divisors}.
\newblock {\em Compos. Math.}, 140(6):1518--1536, 2004.
\newblock MR2098400.

\bibitem{cox_and_combi_2}
J.~Hausen.
\newblock {Cox rings and combinatorics II}.
\newblock {\em Moscow Mathematical Journal}, 8(4):711--757, 2008.

\bibitem{elaines_diss}
E.~Huggenberger.
\newblock Fano Varieties with Torus Action of Complexity One.
\newblock {\em Doctoral Thesis}
\newblock {http://nbn-resolving.de/urn:nbn:de:bsz:21-opus-69570}

\bibitem{quotients_of_toric_varieties}
M.~Kapranov, B.~Sturmfels, and A.~Zelevinsky.
\newblock Quotients of toric varieties.
\newblock {\em Mathematische Annalen}, 290(4):643--655, 1991.

\bibitem{chow_quotients_of_grassmannians}
M.~M. Kapranov.
\newblock Chow quotients of {G}rassmannians. {I}.
\newblock In {\em I. {M}. {G}el{'}fand {S}eminar}, volume~16 of {\em Adv.
  Soviet Math.}, pages 29--110. Amer. Math. Soc., Providence, RI, 1993.
\newblock MR1237834.

\bibitem{losev_manin}
A.~Losev and Y.~Manin.
\newblock New moduli spaces of pointed curves and pencils of flat connections.
\newblock {\em Michigan Math. J.}, 48:443--472, 2000.
\newblock Dedicated to William Fulton on the occasion of his 60th birthday.

\bibitem{shmelkins_invariants}
D.~A. Shmelkin.
\newblock First fundamental theorem for covariants of classical groups.
\newblock {\em Adv. Math.}, 167(2):175--194, 2002.

\bibitem{tropical_grassmannian}
D.~Speyer and B.~Sturmfels.
\newblock The tropical {G}rassmannian.
\newblock {\em Adv. Geom.}, 4(3):389--411, 2004.

\bibitem{subvarieties_of_tori}
J.~Tevelev.
\newblock Compactifications of subvarieties of tori.
\newblock {\em Amer. J. Math.}, 129(4):1087--1104, 2007.

\end{thebibliography}
\end{document}